\documentclass[12pt]{article}
\usepackage{amssymb}
\usepackage{amsmath}
\usepackage{geometry}

\newtheorem{theorem}{Theorem}

\newtheorem{corollary}[theorem]{Corollary}

\newtheorem{definition}[theorem]{Definition}

\newtheorem{lemma}[theorem]{Lemma}

\newtheorem{proposition}[theorem]{Proposition}
\newtheorem{remark}[theorem]{Remark}

\newenvironment{proof}[1][Proof]{\noindent\textbf{#1.} }{\ \rule{0.5em}{0.5em}}
\begin{document}

\begin{center}
{\Large Weighted norm inequalities in Lebesgue spaces with Muckenhoupt
weights and some applications to operators}\textbf{\bigskip }

{\Large Ramazan Akg\"{u}n}$^{1}$\footnotetext[1]{%
The author was supported by Balikesir University Scientific Research Project
2019/61.}\textbf{\bigskip }
\end{center}

\begin{quotation}
\textbf{Abstract }In the present work we give a simple method to obtain
weighted norm inequalities in Lebesgue spaces $L_{p,\gamma }$ with
Muckenhoupt weights $\gamma $. This method is different from celebrated
Extrapolation or Interpolation Theory. In this method starting point is
uniform norm estimates of special form. Then a procedure give desired
weighted norm inequalities in $L_{p,\gamma }.$ We apply this method to
obtain several convolution type inequalities. As an application we consider
a difference operator of type $\Delta _{v}^{r}:=\left( \mathbb{I}-\mathfrak{T%
}_{v}\right) ^{r}$ where $\mathbb{I}$ is the identity operator, $r\in 
\mathbb{N}$ and%
\begin{equation*}
\mathfrak{T}_{v}f\left( x\right) :=\frac{1}{v}\int\nolimits_{x}^{x+v}f\left(
t\right) dt,\quad x\in \left[ -\pi ,\pi \right] ,\text{\quad }v>0,\quad 
\mathfrak{T}_{0}:=\mathbb{I}.
\end{equation*}%
We obtain main properties of $\Delta _{v}^{r}f$ for functions $f$ given in $%
L_{p,\gamma }$, $1\leq p<\infty $, with weights $\gamma $ satisfying the
Muckenhoupt's $A_{p}$ condition. Also we consider some applications of
difference operator $\Delta _{v}^{r}$ in these spaces. In particular, we
obtain that difference $\left\Vert \Delta _{v}^{r}f\right\Vert _{p,\gamma }$
\ is a useful tool for computing the smoothness properties of functions
these spaces. It is obtained that $\left\Vert \Delta _{v}^{r}f\right\Vert
_{p,\gamma }$ is equivalent to Peetre's \textit{K}-functional.


\end{quotation}

\section{Introduction}

Let $\mathbb{T}:=\left[ -\pi ,\pi \right] $. A function $\gamma :\mathbb{%
T\rightarrow }\left[ 0,\infty \right] $ will be called weight if $\gamma $
is measurable and positive a.e. on $\mathbb{T}$. We denote $\gamma \left(
A\right) :=\int\nolimits_{A}\gamma \left( u\right) du$ for a measurable set $%
J\subseteqq \mathbb{T}$. An integrable, 2$\pi $-periodic weight function $%
\gamma ,$ defined on $\mathbb{T},$ satisfies Muckenhoupt's $A_{1}$ condition
(briefly $\gamma \in A_{1}$) if%
\begin{equation*}
\left[ \gamma \right] _{1}:=\sup\limits_{J\subseteqq \mathbb{T}}\frac{\gamma
\left( J\right) }{mes(J)}\underset{x\in J}{ess}\sup \left( \gamma \left(
x\right) \right) ^{-1}<\infty ;
\end{equation*}%
satisfy the Muckenhoupt's $A_{p}$, $1<p<\infty $, condition (briefly $\gamma
\in A_{p}$) if%
\begin{equation*}
\left[ \gamma \right] _{p}:=\sup\limits_{J\subseteqq \mathbb{T}}\frac{\gamma
\left( J\right) }{mes(J)}\left( \frac{1}{mes(J)}\int\nolimits_{J}\left[
\gamma \left( u\right) \right] ^{\frac{-1}{p-1}}du\right) ^{p-1}<\infty 
\text{,\quad }\left( 1<p<\infty \right)
\end{equation*}%
holds with some (Muckenhoupt) constant $\left[ \gamma \right] _{p}$
independent of $J$.

For example, (i) $\left\vert x\right\vert ^{\alpha }\in A_{p}$ iff $%
-1<\alpha <p-1$; (ii) $\left\vert \sin \theta \right\vert ^{\alpha }\in
A_{p} $ iff $-1<\alpha <p-1$;\ $\ $(iii) Densities of harmonic measures in
relation to the Lebesgue surface measure on the boundaries of sufficiently
regular domains satisfy the condition $A_{p}.$ Some other examples are given
e.g. in the article \cite[p.2097]{dyOs80}.

For a weight $\gamma $ on $\mathbb{T}$, we denote by $L_{p,\gamma }$, $1\leq
p<\infty $ the class of real valued measurable functions, defined on $%
\mathbb{T},$ such that 
\begin{equation*}
\int\nolimits_{\mathbb{T}}\left\vert f\left( x\right) \right\vert ^{p}\gamma
\left( x\right) dx<\infty \text{ for }1\leq p<\infty .
\end{equation*}%
For $f\in L_{p,\gamma }$, $1\leq p<\infty $ we set%
\begin{equation*}
\left\Vert f\right\Vert _{p,\gamma }^{p}:=\int\nolimits_{\mathbb{T}%
}\left\vert f\left( x\right) \right\vert ^{p}\gamma \left( x\right) dx.
\end{equation*}%
When $\gamma \equiv 1$ we will set $L_{p}:=L_{p,1}$. Let $C\left( \mathbb{T}%
\right) $ be the collection of continuous functions $f:\mathbb{T}\rightarrow 
\mathbb{R}$ with $\Vert f\Vert _{C\left( \mathbb{T}\right) }:=\max \left\{
\left\vert f\left( x\right) \right\vert :x\in \mathbb{T}\right\} <\infty .$
If $p\in \lbrack 1,\infty )$ and $\gamma \in A_{p}$, then embeddings%
\begin{equation*}
C\left( \mathbb{T}\right) \hookrightarrow L_{p,\gamma }\hookrightarrow L_{1},
\end{equation*}%
hold, because%
\begin{eqnarray}
\left\Vert \cdot \right\Vert _{1} &\leq &\left[ \gamma \right] _{p}^{\frac{1%
}{p}}\left\Vert \gamma \right\Vert _{1}^{-\frac{1}{p}}\left\Vert \cdot
\right\Vert _{p,\gamma }\text{ for }f\in L_{p,\gamma }\text{, and}
\label{boun} \\
\left\Vert \cdot \right\Vert _{p,\gamma } &\leq &\left\Vert \gamma
\right\Vert _{1}^{\frac{1}{p}}\left\Vert \cdot \right\Vert _{\infty }\text{,
for }f\in C\left( \mathbb{T}\right) .  \label{bn2}
\end{eqnarray}

Rest of the paper is organized as follows. The next subsection contain a
simple method to obtain weighted norm inequalities. In section 2 we give
several applications of new method. In subsection 2.1 we give proof of
weighted norm inequalities for convolutions. Also we provide several kernels
fitting our convolution inequality. In subsection 2.2 we obtain boundedness
of the one-sided Steklov operator. In subsection 2.3 we obtain some
properties of Difference Operator based on one-sided Steklov mean. In
subsection 2.4 we define modulus of smoothness and obtain an equivalence
relation with \textit{K}-functional and modulus of smoothness.

\subsection{Method for obtaining weighted norm inequalities}

In this subsection only, notations $\mathbf{c}_{i}$ ($i\in \mathbb{N}$) will
stand for generic positive constants and these can be change in different
places.

To obtain weighted norm inequality%
\begin{equation}
\left\Vert g\right\Vert _{p,\gamma }\leq \mathbf{c}_{1}\left\Vert
f\right\Vert _{p,\gamma }  \label{WNI}
\end{equation}%
for $1\leq p<\infty $, $\gamma \in A_{p}$, $f\in L_{p,\gamma }$, with some
positive constants $\mathbf{c}_{1}$ depending only on $p$ and $\left[ \gamma %
\right] _{p}$, we define an intermediate operator%
\begin{equation*}
F_{f}:\mathbb{T\rightarrow }C\left( \mathbb{T}\right) \text{, \ \ }u\mapsto
F_{f}\left( u\right) \text{ \ \ where }f\in L_{p,\gamma }
\end{equation*}%
such that%
\begin{equation*}
\left\Vert g\right\Vert _{p,\gamma }\leq \mathbf{c}_{2}\left\Vert
F_{g}\left( u\right) \right\Vert _{C\left( \mathbb{T}\right) }\text{ and }%
\left\Vert F_{f}\left( u\right) \right\Vert _{C\left( \mathbb{T}\right)
}\leq \mathbf{c}_{3}\left\Vert f\right\Vert _{p,\gamma }
\end{equation*}%
for some positive constants $\mathbf{c}_{2},\mathbf{c}_{3}$ depending only
on $p$ and $\left[ \gamma \right] _{p}$. Now, if we assume uniform norm
estimate%
\begin{equation}
\left\Vert F_{g}\left( u\right) \right\Vert _{C\left( \mathbb{T}\right)
}\leq \mathbf{c}_{4}\left\Vert F_{f}\left( u\right) \right\Vert _{C\left( 
\mathbb{T}\right) }  \label{UE}
\end{equation}%
holds, then, we obtain desired weighted norm inequality (\ref{WNI}) with $%
\mathbf{c}_{1}=\mathbf{c}_{4}\mathbf{c}_{2}\mathbf{c}_{3}$.

We observe below that, in many concrete situations, (\ref{UE}) is easy to
obtain.

Suppose $1\leq p<\infty $, $\gamma \in A_{p}$, $f\in L_{p,\gamma }$,%
\begin{equation*}
p^{\prime }:=\left\{ 
\begin{tabular}{cc}
$\frac{p}{p-1}$ & for $p>1,$ \\ 
$\infty $ & for $p=1,$%
\end{tabular}%
\right. \quad \gamma ^{\prime }:=\left\{ 
\begin{tabular}{cc}
$\gamma ^{1-p^{\prime }}$ & for $p>1,$ \\ 
$1$ & for $p=1.$%
\end{tabular}%
\right.
\end{equation*}%
For an $G\in L_{p^{\prime },\gamma ^{\prime }}$,\quad $\left\Vert
G\right\Vert _{p^{\prime },\gamma ^{\prime }}\leq 1$ we define%
\begin{equation}
F_{f,G}\left( u\right) :=\int\nolimits_{\mathbb{T}}f\left( x+u\right)
\left\vert G\left( x\right) \right\vert dx,\qquad u\in \mathbb{T}\text{.}
\label{fbb}
\end{equation}

\begin{theorem}
\label{Fu}If $1\leq p<\infty $, $\gamma \in A_{p}$ and $f\in L_{p,\gamma }$,
then the function $F_{f}\left( u\right) $, defined in (\ref{fbb}), is
uniformly continuous on $\mathbb{T}.$
\end{theorem}

\begin{definition}
(\cite[p.96]{HH})Let $\mathbb{N}:\mathbb{=}\left\{ 1,2,3,\cdots \right\} $
be natural numbers and $\mathbb{N}_{0}:=\mathbb{N\cup }\left\{ 0\right\} $.

(a) A family $Q$ of measurable sets $E\subset \mathbb{R}$ is called locally $%
N$-finite ($N\in \mathbb{N}$) if 
\begin{equation*}
\sum_{E\in Q}\chi _{E}\left( x\right) \leq N
\end{equation*}%
almost everywhere in $\mathbb{R}$ where $\chi _{U}$ is the characteristic
function of the set $U$.

(b) A family $Q$ of open bounded sets $U\subset \mathbb{R}$ is locally $1$%
-finite if and only if the sets $U\in Q$ are pairwise disjoint.

(c) Let $U\subset $\textsc{T} be \ a measurable set and%
\begin{equation*}
A_{U}f:=\frac{1}{\left\vert U\right\vert }\int\limits_{U\cap \text{\textsc{T}%
}}\left\vert f\left( t\right) \right\vert dt.
\end{equation*}

(d) For a family $Q$ of open sets $U\subset $\textsc{T} we define averaging
operator by 
\begin{equation*}
T_{Q}:L_{loc}^{1}\left( T\right) \rightarrow L^{0}\left( T\right) ,
\end{equation*}%
\begin{equation*}
T_{Q}f\left( x\right) :=\sum_{U\in Q}\chi _{U\cap \text{\textsc{T}}}\left(
x\right) A_{U}f,\quad x\in \text{\textsc{T}},
\end{equation*}%
where $L^{0}\left( \text{\textsc{T}}\right) $ is the set of measurable
functions on \textsc{T}.
\end{definition}

For a measurable set $A\subset \mathbb{R}$, symbol $\left\vert A\right\vert $
will represent the Lebesgue measure of $A$.

We need a duality result given below. We define $\langle f,g\rangle
:=\int\nolimits_{\mathbb{T}}f(x)g(x)dx$ when integral exists.

\begin{lemma}
\label{dual}(\cite[p.352]{au20})If $1\leq p<\infty $, $\gamma \in A_{p}$,
then, dual of $L_{p,\gamma }$ is $L_{p^{\prime },\gamma ^{\prime }}$ and%
\begin{equation}
\left\Vert f\right\Vert _{p,\gamma }=\underset{G\in L_{p^{\prime },\gamma
^{\prime }}}{\sup }\left\{ \left\vert \langle f,G\rangle \right\vert
:\left\Vert G\right\Vert _{p^{\prime },\gamma ^{\prime }}\leq 1\right\} .
\label{zzz}
\end{equation}
\end{lemma}

Let $c_{0}^{\prime \prime \prime }:=\left\Vert G\right\Vert _{\infty }.$

\begin{definition}
We denote by $S\left( \mathbb{T}\right) $ the collection of simple functions
on $\mathbb{T}$. We set $S_{0}\left( \mathbb{T}\right) :=\left\{ f\in
S\left( \mathbb{T}\right) :f\text{ \ has a compact support in }\mathbb{T}%
\right\} $.
\end{definition}

From Corollary 3.2.14 of \cite[p.79]{DHHR11}, Remark 3.11 of \cite[p.14]{d-h}
and proof of Lemma 6.7 of \cite[p.23]{d-h} we have the following corollary.

\begin{corollary}
\label{pr1}(Corollary 3.2.14 of \cite[p.79]{DHHR11})Let $1\leq p<\infty $,
and $\gamma \in A_{p}$. Then supremum in (\ref{zzz}) is unchanged if we
replace the condition $G\in L_{p^{\prime },\omega ^{\prime }}$ by $G\in
S\left( \mathbb{T}\right) $ or $G\in S_{0}\left( \mathbb{T}\right) .$
\end{corollary}

We define constant $\mathbf{c}_{0}$:=$\max \left\{ \mathbf{c}_{0}^{\prime },%
\left[ \omega \right] _{A_{1}}\right\} $ where%
\begin{equation*}
\mathbf{c}_{0}^{\prime }:=\mathfrak{E}\left( 2\left( 1+\mathbf{c}%
_{0}^{\prime \prime }\right) \left( 1+\pi \right) \right) ^{p}\left(
1+\gamma \left( B\left( 0,1\right) \right) \right) \left( p^{\prime }\left[
\omega \right] _{A_{p}}\right) ^{\frac{1}{p-1}}
\end{equation*}%
and absolute constant $\mathfrak{E}>1$ comes from $p$-Buckley's (when $p>1$)
univariate estimate of Hardy Littlewood maximal function, and $\mathbf{c}%
_{0}^{\prime \prime }:=\left[ \gamma \right] _{p}^{\frac{1}{p}}\left\Vert
\gamma \right\Vert _{1}^{-\frac{1}{p}}.$

\begin{theorem}
\label{Aver}Suppose that $1\leq p<\infty $, $\gamma \in A_{p}$, and $f\in
L_{p,\gamma }$. If $Q$ is $1$-finite family of open bounded subsets of $%
\mathbb{R}$ having Lebesgue measure $1$, then, the averaging operator $T_{Q}$
is uniformly bounded in $L_{p,\gamma }$, namely,%
\begin{equation*}
\left\Vert T_{Q}f\right\Vert _{p,\gamma }\leq \mathbf{c}_{0}\left\Vert
f\right\Vert _{p,\gamma }
\end{equation*}%
holds.
\end{theorem}

In case of $\gamma \equiv 1$ Theorem \ref{Aver} is obtained in variable
exponent Lebesgue spaces by Diening Harjulehto H\"{a}st\"{o} R\r{u}\v{z}i%
\v{c}ka \cite{DHHR11} and in Musielak-Orlicz spaces by Harjulehto H\"{a}st%
\"{o} \cite{HH}.

\begin{definition}
\label{ddd}(\cite{cuf})Let $B$ be a measurable set $B\subseteq \mathbb{T}$, $%
\phi \in L^{1}\left( \mathbb{T}\right) $ and $\int_{\mathbb{T}}\phi \left(
t\right) dt=1$. For each $t>0$ we define $\phi _{t}\left( x\right) =\frac{1}{%
t}\phi \left( \frac{x}{t}\right) $. Such a sequence $\left\{ \phi
_{t}\right\} $ will be called approximate identity. A function%
\begin{equation*}
\tilde{\phi}\left( x\right) =\sup\limits_{\left\vert y\right\vert \geq
\left\vert x\right\vert }\left\vert \phi \left( y\right) \right\vert
\end{equation*}%
will be called radial majorant of $\phi .$ If $\tilde{\phi}\in L^{1}\left( 
\mathbb{T}\right) $, then, sequence $\left\{ \phi _{t}\right\} $ will be
called potential-type approximate identity.
\end{definition}

Using the same proof of of Corollary 4.6.6 of \cite[p.130]{DHHR11} we can
obtain the following theorem.

\begin{theorem}
\label{bt}(Corollary 4.6.6 of \cite[p.130]{DHHR11})Suppose $1\leq p<\infty $%
, $\gamma \in A_{p}$, $f\in L_{p,\gamma }$, and $\phi $ is a potential-type
approximate identity with radial majorant $\tilde{\phi}\in L^{1}\left( 
\mathbb{T}\right) $. Then, for any $t>0$,%
\begin{equation*}
\left\Vert f\ast \phi _{t}\right\Vert _{p,\gamma }\leq C\left\Vert \tilde{%
\phi}\right\Vert _{1}\left\Vert f\right\Vert _{p,\gamma }
\end{equation*}%
and%
\begin{equation*}
\underset{t\rightarrow 0}{\lim }\left\Vert f\ast \phi _{t}-f\right\Vert
_{p,\gamma }=0
\end{equation*}%
hold with a positive constant $C$ depend only on $p,\gamma .$
\end{theorem}

Using Theorem \ref{Aver}, Corollary 4.6.6 of \cite[p.130]{DHHR11} and
Theorem \ref{bt} we have the following proposition.

\begin{proposition}
\label{pr1+}Let $1\leq p<\infty $, and $\gamma \in A_{p}$. Then%
\begin{equation*}
\frac{1}{12\mathbf{c}_{0}}\left\Vert f\right\Vert _{p,\gamma }\leq
\sup_{G\in L_{p^{\prime },\gamma ^{\prime }}\cap C^{\infty }:\left\Vert
G\right\Vert _{p^{\prime },\gamma ^{\prime }}\leq 1}\int_{\mathbb{T}%
}\left\vert f\left( x\right) \right\vert \left\vert G\left( x\right)
\right\vert dx\leq 2\left\Vert f\right\Vert _{p,\gamma }
\end{equation*}%
holds for all $f\in L_{p,\gamma }$.
\end{proposition}

\begin{theorem}[Main Theorem]
\label{tra}Let $1\leq p<\infty $, $\gamma \in A_{p}$, $f$, $g\in L_{p,\gamma
}$. In this case, if inequality%
\begin{equation*}
\left\Vert F_{g,G}\right\Vert _{C\left( \mathbb{T}\right) }\leq \mathbf{c}%
_{5}\left\Vert F_{f,G}\right\Vert _{C\left( \mathbb{T}\right) }
\end{equation*}%
holds for some absolute constant $\mathbf{c}_{5}$, then%
\begin{equation}
\left\Vert g\right\Vert _{p,\gamma }\leq 24\mathbf{c}_{0}\mathbf{c}%
_{0}^{\prime \prime }\mathbf{c}_{0}^{\prime \prime \prime }\mathbf{c}%
_{5}\left\Vert f\right\Vert _{p,\gamma }.  \label{rrr}
\end{equation}
\end{theorem}

As a result Transference Result (TR) we obtain the following result related
to boundedness of translations of Generalized Steklov Mean.

\begin{theorem}
\label{stek} We suppose that $\gamma $ is a 2$\pi $-periodic weight on $%
\mathbb{T}$ so that $\gamma $ belongs to the class $A_{p}$, $1\leq p<\infty $%
. If $0<\lambda <\infty $ and $\tau \in \mathbb{R}$, then family of
translations of Generalized Steklov Mean Operators $\{S_{\lambda ,\tau
}\}_{1\leq \lambda <\infty }$, defined by%
\begin{equation*}
S_{\lambda ,\tau }f(x)=\lambda \int_{x+\tau -1/(2\lambda )}^{x+\tau
+1/(2\lambda )}f(u)du,\qquad x\in \mathbb{T},
\end{equation*}%
is uniformly bounded (in $\lambda $ and $\tau $) in $L_{p,\gamma }$, namely,%
\begin{equation*}
\left\Vert S_{\lambda ,\tau }f\right\Vert _{p,\gamma }\leq 24\mathbf{c}_{0}%
\mathbf{c}_{0}^{\prime \prime }\mathbf{c}_{0}^{\prime \prime \prime
}\left\Vert f\right\Vert _{p,\gamma }
\end{equation*}%
holds for any $\lambda \in \left( 0,\infty \right) $ and $\tau \in \mathbb{R}
$. Also,%
\begin{equation}
\lim_{\lambda ^{-1}+\tau \rightarrow 0}\left\Vert S_{\lambda ,\tau
}f-f\right\Vert _{p,\gamma }=0.  \label{yknsk}
\end{equation}
\end{theorem}

Note that, for $0<v<\infty $, $\lambda :=1/v$ and $\tau =0$ we get%
\begin{equation*}
T_{v}f\left( x\right) :=S_{1/v,0}f\left( x\right) =\frac{1}{v}%
\int\nolimits_{-v/2}^{v/2}f\left( x+t\right) dt
\end{equation*}%
and from (\ref{yknsk})%
\begin{equation}
\left\Vert T_{v}f-f\right\Vert _{p,\gamma }\rightarrow 0\text{, \ \ as }%
v\rightarrow 0+.  \label{gfh}
\end{equation}

We can give proofs of the results, given above, in order.

\begin{proof}[Proof of Theorem \protect\ref{Fu}]
Since $C\left( \mathbb{T}\right) $ is a dense subset of $L_{p,\gamma }$, we
consider functions $f\in C\left( \mathbb{T}\right) $ firstly. Take $%
\varepsilon >0$ and $u_{1},u_{2}\in \mathbb{T}$. By (uniform) continuity of $%
f$ and (\ref{bn2}), there exist $\delta :=\delta \left( \varepsilon \right)
>0$ so that%
\begin{equation*}
\left\vert f\left( \cdot +u_{1}\right) -f\left( \cdot +u_{2}\right)
\right\vert <\frac{\varepsilon }{\mathbf{c}_{0}^{\prime \prime }},
\end{equation*}%
when $\left\vert u_{1}-u_{2}\right\vert <\delta $. Then, for $\left\vert
u_{1}-u_{2}\right\vert <\delta $, $u_{1},u_{2}\in \mathbb{T}$ we have%
\begin{equation*}
\left\vert F_{f,G}\left( u_{1}\right) -F_{f,G}\left( u_{2}\right)
\right\vert \leq \int\nolimits_{\mathbb{T}}\left\vert f\left( x+u_{i}\right)
-f\left( x+u_{2}\right) \right\vert \left\vert G\left( x\right) \right\vert
dx
\end{equation*}%
\begin{equation*}
\leq \underset{x,u_{1},u_{2}\in \mathbb{T}}{\sup }\left\vert f\left(
x+u_{1}\right) -f\left( x+u_{2}\right) \right\vert \left\Vert G\right\Vert
_{1}
\end{equation*}%
\begin{equation*}
\leq \frac{\varepsilon }{\mathbf{c}_{0}^{\prime \prime }}\mathbf{c}%
_{0}^{\prime \prime }\left\Vert G\right\Vert _{p^{\prime },\gamma ^{\prime
}}=\varepsilon
\end{equation*}%
and the conclusion of Theorem \ref{Fu} follows. For the general case $f\in
L_{p,\gamma }$ there exists an $g\in C\left( \mathbb{T}\right) $ so that 
\begin{equation*}
\left\Vert f-g\right\Vert _{p,\gamma }<\xi 4^{-1}\left( \mathbf{c}%
_{0}^{\prime \prime \prime }\mathbf{c}_{0}^{\prime \prime }\right) ^{-1}
\end{equation*}%
for any $\xi >0.$ Then%
\begin{equation*}
\left\vert F_{f,G}\left( u_{1}\right) -F_{f,G}\left( u_{2}\right)
\right\vert =\left\vert F_{f,G}\left( u_{1}\right) -F_{g,G}\left(
u_{1}\right) \right\vert +\left\vert F_{g,G}\left( u_{1}\right)
-F_{g,G}\left( u_{2}\right) \right\vert
\end{equation*}%
\begin{equation*}
+\left\vert F_{g,G}\left( u_{2}\right) -F_{f,G}\left( u_{2}\right)
\right\vert \leq \left\vert F_{f-g,G}\left( u_{1}\right) \right\vert
+\left\vert F_{g-f,G}\left( u_{2}\right) \right\vert +\xi /2
\end{equation*}%
\begin{equation*}
\leq 2\mathbf{c}_{0}^{\prime \prime \prime }\mathbf{c}_{0}^{\prime \prime
}\left\Vert f-g\right\Vert _{p,\gamma }+\xi /2=\xi .
\end{equation*}%
As a result $F_{f}$ is uniformly continuous on $\mathbb{T}.$
\end{proof}

\begin{proof}[Proof of Main Theorem \protect\ref{tra}]
Suppose $1\leq p<\infty $, $\gamma \in A_{p}$. Let $0\leq f$,$g\in
L_{p,\gamma }$. If $\left\Vert g\right\Vert _{p,\gamma }=0$, then the result
(\ref{rrr}) is obvious. So, we assume that $\left\Vert g\right\Vert
_{p,\gamma }>0$. Using hypotesis we get%
\begin{equation*}
\left\Vert F_{g,G}\right\Vert _{C\left( \mathbb{T}\right) }\leq \mathbf{c}%
_{5}\left\Vert F_{f,G}\right\Vert _{C\left( \mathbb{T}\right) }
\end{equation*}%
\begin{equation*}
=\mathbf{c}_{5}\left\Vert \int\nolimits_{\mathbb{T}}f\left( x+u\right)
\left\vert G\left( x\right) \right\vert dx\right\Vert _{C\left( \mathbb{T}%
\right) }=\mathbf{c}_{5}\max_{u\in \mathbb{T}}\int\nolimits_{\mathbb{T}%
}f\left( x+u\right) \left\vert G\left( x\right) \right\vert dx
\end{equation*}%
\begin{equation*}
\leq \mathbf{c}_{5}\max_{u\in \mathbb{T}}\left\Vert f\left( \cdot +u\right)
\right\Vert _{1}\left\Vert G\right\Vert _{\infty }=\mathbf{c}_{5}\left\Vert
f\right\Vert _{1}\mathbf{c}_{0}^{\prime \prime \prime }\leq \mathbf{c}_{5}%
\mathbf{c}_{0}^{\prime \prime }\mathbf{c}_{0}^{\prime \prime \prime
}\left\Vert f\right\Vert _{p,\gamma }.
\end{equation*}

Using Proposition \ref{pr1}, for any $\varepsilon >0$ there exists $G\in
L_{p^{\prime },\gamma ^{\prime }}$ with $\left\Vert G\right\Vert _{p^{\prime
},\gamma ^{\prime }}\leq 1$ such that%
\begin{equation*}
\int\nolimits_{\mathbb{T}}g\left( x\right) \left\vert G\left( x\right)
\right\vert dx\geq \left\Vert g\right\Vert _{p,\gamma }-\varepsilon
\end{equation*}%
and one can find%
\begin{eqnarray*}
\left\Vert F_{g,G}\right\Vert _{C\left( \mathbb{T}\right) } &\geq
&\left\vert F_{g}\left( 0\right) \right\vert \geq \frac{1}{12\mathbf{c}_{0}}%
\int\nolimits_{\mathbb{T}}g\left( x\right) \left\vert G\left( x\right)
\right\vert dx \\
&\geq &\frac{1}{12\mathbf{c}_{0}}\left\Vert g\right\Vert _{p,\gamma
}-\varepsilon .
\end{eqnarray*}%
Now taking limit $\varepsilon \rightarrow 0+$ we have%
\begin{equation*}
\left\Vert F_{g,G}\right\Vert _{C\left( \mathbb{T}\right) }\geq \frac{1}{12%
\mathbf{c}_{0}}\left\Vert g\right\Vert _{p,\gamma }
\end{equation*}%
and hence%
\begin{equation*}
\left\Vert g\right\Vert _{p,\gamma }\leq 12\mathbf{c}_{0}\left\Vert
F_{g,G}\right\Vert _{C\left( \mathbb{T}\right) }\leq 12\mathbf{c}_{0}\mathbf{%
c}_{5}\left\Vert F_{f}\right\Vert _{C\left( \mathbb{T}\right) }\leq \mathbf{c%
}_{5}12\mathbf{c}_{0}\mathbf{c}_{0}^{\prime \prime }\mathbf{c}_{0}^{\prime
\prime \prime }\left\Vert f\right\Vert _{p,\gamma }.
\end{equation*}%
For general case $f$,$g\in L_{p,\gamma }$we have%
\begin{equation*}
\left\Vert g\right\Vert _{p,\gamma }\leq \mathbf{c}_{5}24\mathbf{c}_{0}%
\mathbf{c}_{0}^{\prime \prime }\mathbf{c}_{0}^{\prime \prime \prime
}\left\Vert f\right\Vert _{p,\gamma }.
\end{equation*}
\end{proof}

\begin{proof}[Proof of Theorem \protect\ref{stek}]
Using $F_{S_{\lambda ,\tau }f,G}=S_{\lambda ,\tau }F_{f,G}$, and%
\begin{equation*}
\left\Vert F_{S_{\lambda ,\tau }f,G}\right\Vert _{C\left( \mathbb{T}\right)
}=\left\Vert S_{\lambda ,\tau }F_{f,G}\right\Vert _{C\left( \mathbb{T}%
\right) }\leq \left\Vert F_{f,G}\right\Vert _{C\left( \mathbb{T}\right) }
\end{equation*}%
we conclude from TR that%
\begin{equation*}
\left\Vert S_{\lambda ,\tau }f\right\Vert _{p,\gamma }\leq 24\mathbf{c}_{0}%
\mathbf{c}_{0}^{\prime \prime }\mathbf{c}_{0}^{\prime \prime \prime
}\left\Vert f\right\Vert _{p,\gamma }.
\end{equation*}

We can consider (\ref{yknsk}). Since $C\left( \mathbb{T}\right) $ is a dense
subset of $L_{p,\gamma }$, we consider $f\in C\left( \mathbb{T}\right) $
first. Using (\ref{bn2}) we have%
\begin{equation*}
\left\Vert S_{\lambda ,\tau }f-f\right\Vert _{p,\gamma }\leq \left\Vert
\gamma \right\Vert _{1}^{\frac{1}{p}}\left\Vert S_{\lambda ,\tau
}f-f\right\Vert _{\infty }\rightarrow 0\text{, \ \ as }\lambda ^{-1}+\tau
\rightarrow 0.
\end{equation*}%
Now, for the general case $f\in L_{p,\gamma }$ there exists an $g\in C\left( 
\mathbb{T}\right) $ so that 
\begin{equation*}
\left\Vert f-g\right\Vert _{p,\gamma }<\frac{\varepsilon }{2\left( 1+24%
\mathbf{c}_{0}\mathbf{c}_{0}^{\prime \prime }\mathbf{c}_{0}^{\prime \prime
\prime }\right) }
\end{equation*}%
for any $\varepsilon >0.$ Then, for this $\varepsilon $, there exist an $%
N_{\lambda }\in \mathbb{R}^{+}$ and a $\delta >0$ such that for any $\lambda
\geq N_{\lambda }$ and $\left\vert \tau \right\vert <\delta $ one gets $%
\left\Vert S_{\lambda ,\tau }g-g\right\Vert _{p,\gamma }<\frac{\varepsilon }{%
2}$ and, hence,%
\begin{equation*}
\left\Vert S_{\lambda ,\tau }f-f\right\Vert _{p,\gamma }\leq \left\Vert
S_{\lambda ,\tau }f-S_{\lambda ,\tau }g\right\Vert _{p,\gamma }+\left\Vert
S_{\lambda ,\tau }g-g\right\Vert _{p,\gamma }+\left\Vert g-f\right\Vert
_{p,\gamma }
\end{equation*}%
\begin{equation*}
\leq (1+24\mathbf{c}_{0}\mathbf{c}_{0}^{\prime \prime }\mathbf{c}%
_{0}^{\prime \prime \prime })\left\Vert f-g\right\Vert _{p,\gamma
}+\left\Vert S_{\lambda ,\tau }g-g\right\Vert _{p,\gamma }
\end{equation*}%
\begin{equation*}
\leq \frac{\varepsilon }{2}+\frac{\varepsilon }{2}=\varepsilon .
\end{equation*}
\end{proof}

\section{Applications}

Starting from this section, we will use notations $C_{i}$, $i=1,2,3,...$ for
certain positive constants and these will not change in different places
until the end of the paper. Let $C_{i}^{m}:=\left( C_{i}\right) ^{m}$ for $%
m=1,2,3,....$

Several results are known about specific convolution inequalities in papers 
\cite{Akgeja}, \cite{isrYir16}, \cite{ka}, \cite{sh96}, \cite{vol17}. In the
following section we give a method for weighted convolution inequalities.

\subsection{Weighted Convolution Inequalities}

Let $\lambda \geq 1$, $k_{\lambda }=k_{\lambda }(x)$ be 2$\pi $-periodic
function defined on $\mathbb{T}$, such that

\begin{equation}
\int_{\mathbb{T}}k_{\lambda }(x)dx\leq C_{1}<\infty .  \label{cup}
\end{equation}%
We define the class of operators $K_{\lambda }f\left( x\right)
:=\int\limits_{\mathbb{T}}f(x-t)k_{\lambda }(t)dt$ \ for $\lambda \geq 1$.
Then class of operators $\{K_{\lambda }f\}_{1\leq \lambda <\infty }$ is
uniformly bounded (in $\lambda $) in $L_{p,\gamma }$ for $1\leq p<\infty $
and $\gamma \in A_{p}.$

\begin{theorem}
\label{conv} Let $\lambda >0$, $k_{\lambda }=k_{\lambda }(x)$ be 2$\pi $%
-periodic function defined on $\mathbb{T}$, such that (\ref{cup}) to hold.
We suppose that $\gamma $ is a 2$\pi $-periodic weight on $\mathbb{T}$ so
that $\gamma $ is belong to the class $A_{p}$, $1\leq p<\infty $. Then,

(i) $F_{K_{\lambda }f}\left( \cdot \right) =K_{\lambda }F_{f}\left( \cdot
\right) $ \ and

(ii) the class of operators $\{K_{\lambda }f\}_{1\leq \lambda <\infty }$ is
uniformly bounded (in $\lambda $) in $L_{p,\gamma }$, namely,%
\begin{equation*}
\left\Vert K_{\lambda }f\right\Vert _{p,\gamma }\leq 24\mathbf{c}_{0}\mathbf{%
c}_{0}^{\prime \prime }\mathbf{c}_{0}^{\prime \prime \prime }C_{1}\left\Vert
f\right\Vert _{p,\gamma }.
\end{equation*}
\end{theorem}

Specific examples of kernels satisfying the conditions (\ref{cup}) are among
others, Steklov , Poisson , Ces\`{a}ro , Jackson and Fej\'{e}r kernels.

\begin{proof}[Proof of Theorem \protect\ref{conv}]
(i) For any $u\in \mathbb{T},$%
\begin{eqnarray*}
F_{K_{\lambda }f}\left( u\right) &=&\int\limits_{\mathbb{T}}K_{\lambda
}f\left( x+u\right) G\left( x\right) dx \\
&=&\int\limits_{\mathbb{T}}\int\limits_{\mathbb{T}}f(x+u-t)k_{\lambda
}(t)dtG\left( x\right) dx
\end{eqnarray*}%
\begin{equation*}
=\int\limits_{\mathbb{T}}\int\limits_{\mathbb{T}}f(x+u-t)G\left( x\right)
dxk_{\lambda }(t)dt
\end{equation*}%
\begin{equation*}
=\int\limits_{\mathbb{T}}F_{f}\left( u-t\right) k_{\lambda }(t)dt=K_{\lambda
}F_{f}\left( u\right) .
\end{equation*}

(ii) Clearly%
\begin{equation*}
\left\Vert K_{\lambda }g\right\Vert _{C\left( \mathbb{T}\right) }=\left\Vert
\int\limits_{\mathbb{T}}g(x-t)k_{\lambda }(t)dt\right\Vert _{C\left( \mathbb{%
T}\right) }\leq C_{1}\left\Vert g\right\Vert _{C\left( \mathbb{T}\right) }
\end{equation*}%
for any $g\in C\left( \mathbb{T}\right) $.

Using (i) we obtain%
\begin{equation*}
\left\Vert F_{K_{\lambda }f}\right\Vert _{C\left( \mathbb{T}\right)
}=\left\Vert K_{\lambda }F_{f}\right\Vert _{C\left( \mathbb{T}\right) }\leq
C_{1}\left\Vert F_{f}\right\Vert _{C\left( \mathbb{T}\right) }.
\end{equation*}%
Now Theorem \ref{tra} give%
\begin{equation*}
\left\Vert K_{\lambda }f\right\Vert _{p,\gamma }\leq 24\mathbf{c}_{0}\mathbf{%
c}_{0}^{\prime \prime }\mathbf{c}_{0}^{\prime \prime \prime }C_{1}\left\Vert
f\right\Vert _{p,\gamma }
\end{equation*}%
as required.
\end{proof}

Here we can give several corollaries of Theorem \ref{conv}.

i) \textsc{Steklov Operator}: Let $\Delta _{\lambda }:=\left[ -1/(2\lambda
),1/(2\lambda )\right] $, $\lambda \geq 1$ and%
\begin{equation*}
k_{\lambda }(x):=\left\{ 
\begin{tabular}{ll}
$\lambda $ & , when $x\in \Delta _{\lambda },$ \\ 
$0$ & , when $x\in \mathbb{T}\setminus \Delta _{\lambda }.$%
\end{tabular}%
\right.
\end{equation*}%
We extend $k_{\lambda }$ to $\mathbb{R}:\mathbb{=}(-\infty ,\infty )$ with
period 2$\pi $. In this case Steklov operator $S_{\lambda }f$ is represented
as%
\begin{equation*}
S_{\lambda }f(x)=\int_{\mathbb{T}}f(x-t)k_{\lambda }(t)dt=\lambda
\int_{x-1/(2\lambda )}^{x+1/(2\lambda )}f(u)du.
\end{equation*}%
Since%
\begin{equation*}
\int_{\mathbb{T}}k_{\lambda }(x)dx=1
\end{equation*}%
we obtain the following corollary.

\begin{corollary}
If $\gamma $ is a 2$\pi $-periodic weight on $\mathbb{T}$ so that $\gamma $
is belong to the class $A_{p}$, $1\leq p<\infty $, then the sequence of
Steklov Operators $\{S_{\lambda }f\}_{1\leq \lambda <\infty }$ is uniformly
bounded (in $\lambda $) in $L_{p,\gamma }$, namely,%
\begin{equation*}
\left\Vert S_{\lambda }f\right\Vert _{p,\gamma }\leq 24\mathbf{c}_{0}\mathbf{%
c}_{0}^{\prime \prime }\mathbf{c}_{0}^{\prime \prime \prime }\left\Vert
f\right\Vert _{p,\gamma }.
\end{equation*}
\end{corollary}

Some results of this type can be found in \cite{ak18}, \cite{istes16}, \cite%
{prSvSc17}.

2) \textsc{Jackson Operator}: Let $n\in \mathbb{N}$ and $\mathcal{T}_{n}$%
\textit{\ }be the class of real trigonometric polynomials of degree not
greater than $n$. Let%
\begin{equation}
D_{n}f(x):=\frac{1}{\pi }\int\nolimits_{\mathbb{T}}f(x-t)J_{2,\lfloor \frac{n%
}{2}\rfloor +1}(t)dt\in \mathcal{T}_{n}  \label{Dn}
\end{equation}%
be the Jackson operator (polynomial) where $J_{2,n}$ is the Jackson kernel%
\begin{equation*}
J_{2,n}(x):=\frac{1}{\varkappa _{2,n}}\left( \frac{\sin (nx/2)}{\sin (x/2)}%
\right) ^{4},\quad \varkappa _{2,n}:=\frac{1}{\pi }\int\nolimits_{\mathbb{%
-\pi }}^{\pi }\left( \frac{\sin (nt/2)}{\sin (t/2)}\right) ^{4}dt.
\end{equation*}%
It is known that (\cite[p.147]{DzSh})%
\begin{equation*}
2^{-3/2}n^{3}3\leq \varkappa _{2,n}\leq 2^{-3/2}n^{3}5
\end{equation*}%
and $J_{2,n}$ satisfies relations%
\begin{equation*}
\left. 
\begin{array}{c}
\frac{1}{\pi }\int_{\mathbb{T}}J_{2,n}(u)du=1;\quad \left\vert
J_{2,n}(u)\right\vert \mathbb{\leq }\frac{2\sqrt{2}}{3}\pi ^{4},\quad
n^{-3/4}\leq u\leq \pi ,\medskip \\ 
\max_{u\in \mathbb{T}}\left\vert J_{2,n}(u)\right\vert \mathbb{\leq }\left( 
\frac{\pi }{2}\right) ^{4}n;\quad \frac{1}{\pi }\int_{0}^{\pi }uJ_{2,n}(u)du%
\mathbb{\leq }\frac{5}{2n},\medskip%
\end{array}%
\right\}
\end{equation*}%
and property (\ref{cup}) is satisfied.

Now, Theorem \ref{conv} gives

\begin{corollary}
\label{unif}If $\gamma $ is a 2$\pi $-periodic weight on $\mathbb{T}$ so
that $\gamma $ is belong to the class $A_{p}$, $1\leq p<\infty $, then the
sequence of Jackson operators $\{D_{n}f\}_{1\leq n<\infty }$ is uniformly
bounded (in $n$) in $L_{p,\gamma }$, namely,%
\begin{equation*}
\left\Vert D_{n}f\right\Vert _{p,\gamma }\leq \pi 24\mathbf{c}_{0}\mathbf{c}%
_{0}^{\prime \prime }\mathbf{c}_{0}^{\prime \prime \prime }\left\Vert
f\right\Vert _{p,\gamma }.
\end{equation*}
\end{corollary}

3) \textsc{Fejer Operator}: We suppose that $\gamma $ is a 2$\pi $-periodic
weight on $\mathbb{T}$ so that $\gamma $ is belong to the class $A_{p}$, $%
1\leq p<\infty $. Let%
\begin{equation}
f\left( x\right) \backsim \frac{a_{0}\left( f\right) }{2}+\sum%
\limits_{k=1}^{\infty }\left( a_{k}\left( f\right) \cos kx+b_{k}\left(
f\right) \sin kx\right) =:\sum\limits_{k=0}^{\infty }A_{k}\left( x,f\right)
\label{aa}
\end{equation}%
be the Fourier series of $f\in W_{p,\gamma }^{1}$ and $S_{n}\left( f\right)
:=S_{n}\left( x,f\right) :=\sum\limits_{k=0}^{n}A_{k}\left( x,f\right)
,\quad n=0,1,2,\ldots $ be the partial sum of the Fourier series (\ref{aa}).
Fejer Operator (the first arithmetic mean) is defined as%
\begin{equation*}
F_{n}(f,\cdot ):=\frac{1}{n+1}\left\{ S_{0}(\cdot ,f)+S_{1}(\cdot
,f)+...+S_{n}(\cdot ,f)\right\}
\end{equation*}%
with Fej\'{e}r kernel%
\begin{equation*}
k_{n}\left( u\right) =\frac{2}{\left( n+1\right) }\left[ \frac{\sin \left(
\left( n+1\right) u/2\right) }{\sin \left( u/2\right) }\right] ^{2}
\end{equation*}%
satisfies (\ref{cup}) since%
\begin{equation*}
\frac{1}{\pi }\int_{\mathbb{T}}k_{n}(u)du=1,\quad \max_{t\in \mathbb{T}%
}k_{n}(t)\leq n+1.
\end{equation*}%
If we take $k_{\lambda }\left( u\right) =k_{n}\left( u\right) $ for $n\leq
\lambda <n+1$ we have%
\begin{equation*}
F_{\lambda }(f,x)=\frac{1}{\pi }\int\nolimits_{\mathbb{T}}f(x-t)k_{\lambda
}\left( t\right) (t)dt.
\end{equation*}

\begin{corollary}
If $\gamma $ is a 2$\pi $-periodic weight on $\mathbb{T}$ so that $\gamma $
is belong to the class $A_{p}$, $1\leq p<\infty $, then $F_{\lambda }f$ is
uniformly bounded (in $\lambda $) in $L_{p,\gamma }$, namely,%
\begin{equation*}
\left\Vert F_{\lambda }(f,\cdot )\right\Vert _{p,\gamma }\leq \pi 24\mathbf{c%
}_{0}\mathbf{c}_{0}^{\prime \prime }\mathbf{c}_{0}^{\prime \prime \prime
}\left\Vert f\right\Vert _{p,\gamma },\quad \forall \lambda \geq 1\text{.}
\end{equation*}
\end{corollary}

\begin{corollary}
\label{berns}(Weighted Bernstein's Inequality) Let $n,r\in \mathbb{N}$, $%
1\leq p<\infty $, $\gamma \in A_{p}$ and $U_{n}\in \mathcal{T}_{n}$. Then%
\begin{equation}
\left\Vert U_{n}^{(r)}\right\Vert _{p,\gamma }\leq \pi 48\mathbf{c}_{0}%
\mathbf{c}_{0}^{\prime \prime }\mathbf{c}_{0}^{\prime \prime \prime
}n^{r}\left\Vert U_{n}\right\Vert _{p,\gamma }.  \label{ber}
\end{equation}
\end{corollary}

\begin{proof}[Proof of Corollary \protect\ref{berns}]
Since%
\begin{equation*}
\left\vert U_{n}^{\prime }(\cdot )\right\vert \leq 2nF_{n-1}(\cdot
,\left\vert U_{n}\right\vert ),
\end{equation*}%
we obtain weighted Bernstein's inequality%
\begin{equation*}
\left\Vert U_{n}^{\prime }\right\Vert _{p,\gamma }\leq \pi 48\mathbf{c}_{0}%
\mathbf{c}_{0}^{\prime \prime }\mathbf{c}_{0}^{\prime \prime \prime
}n\left\Vert U_{n}\right\Vert _{p,\gamma }.
\end{equation*}%
Then%
\begin{equation*}
\left\Vert U_{n}^{(r)}\right\Vert _{p,\gamma }\leq \pi 48\mathbf{c}_{0}%
\mathbf{c}_{0}^{\prime \prime }\mathbf{c}_{0}^{\prime \prime \prime
}n^{r}\left\Vert U_{n}\right\Vert _{p,\gamma }
\end{equation*}%
for $r\in \mathbb{N}.$
\end{proof}

4) \textsc{Dela Vallee Poussin Operator}: We define, for $n\in \mathbb{N}%
\cup \{0\}$, De la Vall\'{e}e-Poussin mean as%
\begin{equation*}
V_{n}(f,\cdot )=\frac{1}{n}\sum\limits_{i=0}^{n-1}S_{n+i}(\cdot ,f).
\end{equation*}%
Since%
\begin{equation*}
V_{n}(f,\cdot )=2F_{2n-1}(f,\cdot )-F_{n-1}(f,\cdot )
\end{equation*}%
Theorem \ref{conv} gives the following corollary.

\begin{corollary}
\label{VP}If $\gamma $ is a 2$\pi $-periodic weight on $\mathbb{T}$ so that $%
\gamma $ is belong to the class $A_{p}$, $1\leq p<\infty $, then $V_{n}f$ is
uniformly bounded (in $n$) in $L_{p,\gamma }$, namely,%
\begin{equation*}
\left\Vert V_{n}f\right\Vert _{p,\gamma }\leq 72\pi \mathbf{c}_{0}\mathbf{c}%
_{0}^{\prime \prime }\mathbf{c}_{0}^{\prime \prime \prime }\left\Vert
f\right\Vert _{p,\gamma },\quad \forall n\in \mathbb{N}\text{.}
\end{equation*}
\end{corollary}

\subsection{One Sided Steklov Operator}

In this section we will consider the uniform boundedness of the family of
one sided Steklov Operators $\mathfrak{T}_{v}$ defined below. Let

\begin{equation*}
\mathfrak{T}_{v}f\left( x\right) :=\frac{1}{v}\int\nolimits_{x}^{x+v}f\left(
t\right) dt,\quad x\in \mathbb{T},\quad v\in \left( 0,1\right) \text{ and }%
\mathfrak{T}_{0}f:=f.
\end{equation*}%
In this case, $0<v<\infty $, $\lambda :=1/v$ and $\tau =0$ then,%
\begin{equation*}
S_{1/v,v/2}f\left( \cdot \right) =\frac{1}{v}\int\nolimits_{0}^{v}f\left(
\cdot +t\right) dt=\mathfrak{T}_{v}f\left( \cdot \right) .
\end{equation*}%
Hence we get the following theorem.

\begin{theorem}
\label{stekl op norm}We suppose that $\gamma $ is a 2$\pi $-periodic weight
on $\mathbb{T}$ so that $\gamma $ belongs to the class $A_{p}$, $1\leq
p<\infty $. Then the class of operators $\{\mathfrak{T}_{v}f\}_{0<v<\infty }$
is uniformly bounded (in $v$) in $L_{p,\gamma },$ namely,%
\begin{equation*}
\left\Vert \mathfrak{T}_{v}f\right\Vert _{p,\gamma }\leq 24\mathbf{c}_{0}%
\mathbf{c}_{0}^{\prime \prime }\mathbf{c}_{0}^{\prime \prime \prime
}\left\Vert f\right\Vert _{p,\gamma },
\end{equation*}%
\begin{equation*}
\left\Vert \mathfrak{T}_{v}f-f\right\Vert _{p,\gamma }\rightarrow 0,\quad 
\text{as }v\rightarrow 0^{+}.
\end{equation*}
\end{theorem}

We consider the operator (\cite{Ky97,sh13i}), defined for $f\in L_{p,\gamma
} $, $\gamma \in A_{p}$, $1\leq p<\infty $,%
\begin{equation*}
\left( R_{v}f\right) \left( x\right) :=\frac{2}{v}\int\nolimits_{v/2}^{v}%
\left( \frac{1}{h}\int\nolimits_{0}^{h}f\left( x+t\right) dt\right) dh,\quad
x\in \mathbb{T}\text{,}\quad 0<v<\infty .
\end{equation*}%
As a corollary of Theorem \ref{stekl op norm} it is easily seen that

\begin{corollary}
If $1\leq p<\infty $, $0<v<\infty $, $\gamma \in A_{p}$, and $f\in
L_{p,\gamma }$, then%
\begin{equation*}
\left\Vert R_{v}f\right\Vert _{p,\gamma }\leq 24\mathbf{c}_{0}\mathbf{c}%
_{0}^{\prime \prime }\mathbf{c}_{0}^{\prime \prime \prime }\left\Vert
f\right\Vert _{p,\gamma }.
\end{equation*}
\end{corollary}

\subsection{Difference Operator}

After the last subsection we can define Difference Operator%
\begin{equation*}
\Delta _{v}^{k}:=\left( \mathbb{I}-\mathfrak{T}_{v}\right) ^{k}\text{ for }%
k\in \mathbb{N}\text{ and }0<v\leq 1.
\end{equation*}%
We will consider the main properties of this Difference Operator $\Delta
_{v}^{k}$. Let $\mathfrak{T}_{v}^{i}:=\mathfrak{T}_{v}\left( \mathfrak{T}%
_{v}^{i-1}\right) $ for $i\in \mathbb{N}$ and let $\mathfrak{T}_{v}^{0}:=%
\mathbb{I}$. By Theorem \ref{stekl op norm} we have%
\begin{equation*}
\left\Vert (\mathbb{I}-\mathfrak{T}_{v})^{k}f\right\Vert _{p,\gamma }\leq
\left( 1+24\mathbf{c}_{0}\mathbf{c}_{0}^{\prime \prime }\mathbf{c}%
_{0}^{\prime \prime \prime }\right) ^{k}\left\Vert f\right\Vert _{p,\gamma }.
\end{equation*}%
Difference operators generally closely related with the smoothness of the
given function.

\textbf{Definition }For $k\in \mathbb{N}$ we define the \textit{modulus of
smoothness }of $f\in L_{p,\gamma }$, $1\leq p<\infty $, $\gamma \in A_{p}$,
as%
\begin{equation}
\Omega _{k}\left( f,v\right) _{p,\gamma }:=\left\Vert \left( \mathbb{I}-%
\mathfrak{T}_{v}\right) ^{k}f\right\Vert _{p,\gamma },\quad v>0.  \label{nm}
\end{equation}

\textbf{Definition }Let $X$ be a Banach space with norm $\left\Vert \cdot
\right\Vert _{X}$. (i) By $X^{r}$ we denote the class of functions $f\in X$
such that $f^{\left( r-1\right) }$ is absolutely continuous and $f^{\left(
r\right) }\in X$. When $r\in \mathbb{N}$, $1\leq p<\infty $, $\gamma \in
A_{p}$ and $X=L_{p,\gamma }$ we will denote $W_{p,\gamma }^{r}:=X^{r}$.

\textbf{(ii) }We define Peetre's \textit{K}-functional%
\begin{equation*}
K_{r}\left( f,v,X\right) _{X}:=\inf\limits_{g\in X^{r}}\left\{ \left\Vert
f-g\right\Vert _{X}+v^{r}\left\Vert g^{\left( r\right) }\right\Vert
_{X}\right\} \text{,\quad }v>0,
\end{equation*}%
and $K_{r}\left( f,v,p,\gamma \right) :=K_{r}\left( f,v,L_{p,\gamma }\right)
_{L_{p,\gamma }}$ for $r\in \mathbb{N}$, $1\leq p<\infty $, $\gamma \in
A_{p} $, $v>0$ and $f\in L_{p,\gamma }$.

\begin{lemma}
\label{lem1} Let $1\leq p<\infty $, $\gamma \in A_{p}$, $k\in \mathbb{N}$,
and $f\in W_{p,\gamma }^{1}$ be given. Then%
\begin{equation*}
\left\Vert \left( \mathbb{I}-\mathfrak{T}_{v}\right) ^{k}f\right\Vert
_{p,\gamma }\leq 12\mathbf{c}_{0}\mathbf{c}_{0}^{\prime \prime }\mathbf{c}%
_{0}^{\prime \prime \prime }v\left\Vert \left( \mathbb{I}-\mathfrak{T}%
_{v}\right) ^{k-1}f^{\prime }\right\Vert _{p,\gamma }\text{,}\quad 0\leq
v<\infty 
\end{equation*}%
holds.
\end{lemma}

See also \cite{asc19},\cite{ia08}, \cite{sc}.

\begin{corollary}
\label{lem11} Let $1\leq p<\infty $, $\gamma \in A_{p},$ $k\in \mathbb{N}$,
and $f\in W_{p,\gamma }^{k}$ be given. Then%
\begin{equation*}
\left\Vert \left( \mathbb{I}-\mathfrak{T}_{v}\right) ^{k}f\right\Vert
_{p,\gamma }\leq (12\mathbf{c}_{0}\mathbf{c}_{0}^{\prime \prime }\mathbf{c}%
_{0}^{\prime \prime \prime })^{k}v^{k}\left\Vert f^{\left( k\right)
}\right\Vert _{p,\gamma }\text{,}\quad 0\leq v<\infty .
\end{equation*}
\end{corollary}

\begin{lemma}
\label{DI}For $v>0$, $f\in C\left( \mathbb{T}\right) $, $g\in C^{2}\left( 
\mathbb{T}\right) $, $h\in C^{1}\left( \mathbb{T}\right) $ the following
inequalities%
\begin{equation*}
\left\Vert \mathfrak{T}_{v}f\right\Vert _{C\left( \mathbb{T}\right) }\leq
\left\Vert f\right\Vert _{C\left( \mathbb{T}\right) },
\end{equation*}%
\begin{equation*}
\left\Vert \frac{d}{dx}\mathfrak{T}_{v}f\left( x\right) \right\Vert
_{C\left( \mathbb{T}\right) }\leq \frac{2}{v}\left\Vert f\right\Vert
_{C\left( \mathbb{T}\right) },
\end{equation*}%
\begin{equation*}
\left\Vert \left( \frac{d}{dx}\right) ^{2}\mathfrak{T}_{v}f\left( x\right)
\right\Vert _{C\left( \mathbb{T}\right) }\leq \frac{2}{v}\left\Vert \frac{d}{%
dx}T_{v}f\right\Vert _{C\left( \mathbb{T}\right) },
\end{equation*}%
\begin{equation*}
\left\Vert g\left( x\right) -\mathfrak{T}_{v}g\left( x\right) +\frac{v}{2}%
\frac{d}{dx}g\left( x\right) \right\Vert _{C\left( \mathbb{T}\right) }\leq 
\frac{v^{2}}{6}\left\Vert \frac{d^{2}}{dx^{2}}g\right\Vert _{C\left( \mathbb{%
T}\right) },
\end{equation*}%
\begin{equation*}
\left\Vert h-\mathfrak{T}_{v}h\right\Vert _{C\left( \mathbb{T}\right) }\leq 
\frac{v}{2}\left\Vert h^{\prime }\right\Vert _{C\left( \mathbb{T}\right) },
\end{equation*}%
\begin{equation}
\left( 1/36\right) K_{1}\left( f,v,C\left( \mathbb{T}\right) \right)
_{C\left( \mathbb{T}\right) }\leq \left\Vert \left( \mathbb{I}-\mathfrak{T}%
_{v}\right) f\right\Vert _{C\left( \mathbb{T}\right) }\leq 2K_{1}\left(
f,v,C\left( \mathbb{R}\right) \right) _{C\left( \mathbb{T}\right) }
\label{eqA}
\end{equation}%
are hold.
\end{lemma}

Inequalities (\ref{eqA}) give

\begin{corollary}
\label{DIcor} If $0<h\leq v\leq 1$ and $f\in C\left[ \mathbb{T}\right] ,$
then%
\begin{equation}
\left\Vert \left( \mathbb{I}-\mathfrak{T}_{h}\right) f\right\Vert _{C\left[ 
\mathbb{T}\right] }\leq 72\left\Vert \left( \mathbb{I}-\mathfrak{T}%
_{v}\right) f\right\Vert _{C\left[ \mathbb{T}\right] }.  \label{eq1}
\end{equation}
\end{corollary}

\begin{lemma}
\label{bukun} Let $0<h\leq v\leq 1$, $\gamma \in A_{p}$, $1\leq p<\infty $
and $f\in L_{p,\gamma }$. Then%
\begin{equation}
\left\Vert \left( \mathbb{I}-\mathfrak{T}_{h}\right) f\right\Vert _{p,\gamma
}\leq 1728\mathbf{c}_{0}\mathbf{c}_{0}^{\prime \prime }\mathbf{c}%
_{0}^{\prime \prime \prime }\left\Vert \left( \mathbb{I}-\mathfrak{T}%
_{v}\right) f\right\Vert _{p,\gamma }.  \label{bukunn}
\end{equation}
\end{lemma}

\begin{theorem}
\label{kubu} Let $1\leq p<\infty $ and $\gamma \in A_{p}$, $f$, $g\in
L_{p,\gamma }$, $v>0$ and $k\in \mathbb{N}$. Then%
\begin{equation}
\lim\limits_{v\rightarrow 0^{+}}\Omega _{k}\left( f,v\right) _{p,\gamma }=0,
\label{n}
\end{equation}%
\begin{equation}
\Omega _{r+k}\left( f,v\right) _{p,\gamma }\leq \left( 1+24\mathbf{c}_{0}%
\mathbf{c}_{0}^{\prime \prime }\mathbf{c}_{0}^{\prime \prime \prime }\right)
^{r}\Omega _{k}\left( f,v\right) _{p,\gamma }.  \label{bk}
\end{equation}
\end{theorem}

We can give proofs of results given in this subsection.

\begin{proof}[Proof of Lemma \protect\ref{lem1}]
Let $k=1$. Since%
\begin{equation*}
\left( \mathbb{I}-\mathfrak{T}_{v}\right) f\left( x\right) =\frac{1}{v}%
\int\nolimits_{0}^{v}\left( f\left( x\right) -f\left( x+t\right) \right) dt=%
\frac{-1}{v}\int\nolimits_{0}^{v}\int\nolimits_{x}^{x+t}f^{\prime }\left(
s\right) dsdt\text{,}
\end{equation*}%
using generalized Minkowski's inequality for integrals and uniformly
boundedness of $\mathfrak{T}_{v}$ we get%
\begin{equation*}
\left\Vert \left( \mathbb{I}-\mathfrak{T}_{v}\right) f\right\Vert _{p,\gamma
}=\left\Vert \frac{1}{v}\int\nolimits_{0}^{v}\int\nolimits_{x}^{x+t}f^{%
\prime }\left( s\right) dsdt\right\Vert _{p,\gamma }=\left\Vert \frac{1}{v}%
\int\nolimits_{0}^{v}t\frac{1}{t}\int\nolimits_{0}^{t}f^{\prime }\left(
x+s\right) dsdt\right\Vert _{p,\gamma }
\end{equation*}%
\begin{equation*}
=\left\Vert \frac{1}{v}\int\nolimits_{0}^{v}t\mathfrak{T}_{t}f^{\prime
}\left( x\right) dt\right\Vert _{p,\gamma }\leq \frac{1}{v}%
\int\nolimits_{0}^{v}t\left\Vert \mathfrak{T}_{t}f^{\prime }\right\Vert
_{p,\gamma }dt
\end{equation*}%
\begin{equation*}
\leq 24\mathbf{c}_{0}\mathbf{c}_{0}^{\prime \prime }\mathbf{c}_{0}^{\prime
\prime \prime }\tfrac{\left\Vert f^{\prime }\right\Vert _{p,\gamma }}{v}%
\int\nolimits_{0}^{v}tdt\leq 12\mathbf{c}_{0}\mathbf{c}_{0}^{\prime \prime }%
\mathbf{c}_{0}^{\prime \prime \prime }v\left\Vert f^{\prime }\right\Vert
_{p,\gamma }.
\end{equation*}%
Let $k\geq 2$ and set $g\left( \cdot \right) :=\left( \mathbb{I}-\mathfrak{T}%
_{v}\right) ^{k-1}f\left( \cdot \right) $. Then%
\begin{equation*}
\left( \mathbb{I}-\mathfrak{T}_{v}\right) g\left( \cdot \right) =\left( 
\mathbb{I}-\mathfrak{T}_{v}\right) ^{k}f\left( x\right) 
\end{equation*}%
and%
\begin{equation*}
\left( \mathbb{I}-\mathfrak{T}_{v}\right) ^{k}f\left( x\right) =\frac{1}{v}%
\int\nolimits_{0}^{v}\left( g\left( x\right) -g\left( x+t\right) \right) dt=%
\frac{-1}{v}\int\nolimits_{0}^{v}\int\nolimits_{0}^{t}g^{\prime }\left(
x+s\right) dsdt\text{.}
\end{equation*}%
Therefore,%
\begin{eqnarray*}
\left\Vert \left( \mathbb{I}-\mathfrak{T}_{v}\right) ^{k}f\right\Vert
_{p,\gamma } &=&\left\Vert \left( \mathbb{I}-\mathfrak{T}_{v}\right)
g\right\Vert _{p,\gamma }\leq 12\mathbf{c}_{0}\mathbf{c}_{0}^{\prime \prime }%
\mathbf{c}_{0}^{\prime \prime \prime }v\left\Vert g^{\prime }\right\Vert
_{p,\gamma } \\
&=&12\mathbf{c}_{0}\mathbf{c}_{0}^{\prime \prime }\mathbf{c}_{0}^{\prime
\prime \prime }v\left\Vert \left( \mathbb{I}-\mathfrak{T}_{v}\right)
^{k-1}f^{\prime }\right\Vert _{p,\gamma }.
\end{eqnarray*}
\end{proof}

\begin{proof}[Proof of Lemma \protect\ref{DI}]
If $f\in C\left( \mathbb{T}\right) $ then it is clear from the definition of 
$T_{\delta }$ that%
\begin{equation}
\left\Vert \mathfrak{T}_{v}f\right\Vert _{C\left( \mathbb{T}\right) }\leq
\left\Vert f\right\Vert _{C\left( \mathbb{T}\right) }  \label{bir}
\end{equation}%
holds. On the other hand, for $f\in C\left( \mathbb{T}\right) $ we have%
\begin{eqnarray}
\left\Vert \frac{d}{dx}\mathfrak{T}_{v}f\left( x\right) \right\Vert
_{C\left( \mathbb{T}\right) } &=&\left\Vert \frac{d}{dx}\frac{1}{v}%
\int\nolimits_{0}^{v}f\left( x+t\right) dt\right\Vert _{C\left( \mathbb{T}%
\right) }=\left\Vert \frac{1}{v}\frac{d}{dx}\int\nolimits_{0}^{v}f\left(
x+t\right) dt\right\Vert _{C\left( \mathbb{T}\right) }  \notag \\
&=&\left\Vert \frac{1}{v}\left( f\left( x+v\right) -f\left( x\right) \right)
\right\Vert _{C\left( \mathbb{T}\right) }\leq \frac{2}{v}\left\Vert
f\right\Vert _{C\left( \mathbb{T}\right) }.  \label{iki}
\end{eqnarray}%
Inequality (\ref{iki}) also implies%
\begin{equation*}
\left\Vert \left( \frac{d}{dx}\right) ^{2}\mathfrak{T}_{v}f\left( x\right)
\right\Vert _{C\left( \mathbb{T}\right) }\leq \frac{2}{v}\left\Vert \frac{d}{%
dx}\mathfrak{T}_{v}f\right\Vert _{C\left( \mathbb{T}\right) }
\end{equation*}%
for $f\in C\left( \mathbb{T}\right) $. Also for $f\in C^{2}\left( \mathbb{T}%
\right) $%
\begin{equation}
\left\Vert f\left( x\right) -\mathfrak{T}_{v}f\left( x\right) +\frac{v}{2}%
\frac{d}{dx}f\left( x\right) \right\Vert _{C\left( \mathbb{T}\right) }\leq 
\frac{v^{2}}{6}\left\Vert \frac{d^{2}}{dx^{2}}f\right\Vert _{C\left( \mathbb{%
T}\right) }.  \label{vor}
\end{equation}%
To obtain (\ref{vor}) we will use the Taylor formula%
\begin{equation*}
f\left( x+t\right) =f(x)+t\frac{d}{dx}f(x)+\frac{t^{2}}{2}\frac{d^{2}}{dx^{2}%
}f(\xi )
\end{equation*}%
for some $\xi \leq \left[ x,x+t\right] $. Then integrating the last equation
with respect to $t$%
\begin{equation*}
\frac{1}{v}\int\nolimits_{0}^{v}f\left( x+t\right) dt=f(x)+\frac{1}{v}%
\int\nolimits_{0}^{v}tdt\frac{d}{dx}f(x)+\frac{1}{2}\frac{1}{v}%
\int\nolimits_{0}^{v}t^{2}dt\frac{d^{2}}{dx^{2}}f(\xi ),
\end{equation*}%
\begin{equation*}
\mathfrak{T}_{v}f\left( x\right) =f(x)+\frac{v}{2}\frac{d}{dx}f\left(
x\right) +\frac{v^{2}}{6}\frac{d^{2}}{dx^{2}}f(\xi )
\end{equation*}%
and (\ref{vor}) holds.

From $f\in C^{1}\left( \mathbb{T}\right) $ one can obtain%
\begin{equation*}
\left\Vert f-\mathfrak{T}_{v}f\right\Vert _{C\left( \mathbb{T}\right) }\leq 
\frac{v}{2}\left\Vert f^{\prime }\right\Vert _{C\left( \mathbb{T}\right) }.
\end{equation*}%
Indeed, using%
\begin{equation*}
\left( \mathbb{I}-\mathfrak{T}_{v}\right) f\left( x\right) =\frac{1}{v}%
\int\nolimits_{0}^{v}\left( f\left( x\right) -f\left( x+t\right) \right) dt=%
\frac{-1}{v}\int\nolimits_{0}^{v}\int\nolimits_{x}^{x+t}f^{\prime }\left(
s\right) dsdt\text{,}
\end{equation*}%
Generalized Minkowski's inequality for integrals and (\ref{bir}) we get%
\begin{equation*}
\left\Vert \left( \mathbb{I}-\mathfrak{T}_{v}\right) f\right\Vert _{C\left( 
\mathbb{T}\right) }=\left\Vert \frac{1}{v}\int\nolimits_{0}^{v}\int%
\nolimits_{x}^{x+t}f^{\prime }\left( s\right) dsdt\right\Vert _{C\left( 
\mathbb{T}\right) }=\left\Vert \frac{1}{v}\int\nolimits_{0}^{v}t\frac{1}{t}%
\int\nolimits_{0}^{t}f^{\prime }\left( x+s\right) dsdt\right\Vert _{C\left( 
\mathbb{T}\right) }
\end{equation*}%
\begin{equation*}
=\left\Vert \frac{1}{v}\int\nolimits_{0}^{v}t\mathfrak{T}_{t}f^{\prime
}\left( x\right) dt\right\Vert _{C\left( \mathbb{T}\right) }\leq \frac{1}{v}%
\int\nolimits_{0}^{v}t\left\Vert \mathfrak{T}_{t}f^{\prime }\right\Vert
_{C\left( \mathbb{T}\right) }dt
\end{equation*}%
\begin{equation*}
\leq \left\Vert f^{\prime }\right\Vert _{C\left( \mathbb{T}\right) }\tfrac{1%
}{v}\int\nolimits_{0}^{v}tdt\leq 2^{-1}v\left\Vert f^{\prime }\right\Vert
_{C\left( \mathbb{T}\right) }.
\end{equation*}

Now (\ref{bir}), (\ref{iki}) and (\ref{vor}) imply that%
\begin{equation}
\left( 1/36\right) K_{1}\left( f,v,C\left( \mathbb{T}\right) \right)
_{C\left( \mathbb{T}\right) }\leq \left\Vert \left( \mathbb{I}-T_{v}\right)
f\right\Vert _{C\left( \mathbb{T}\right) }\leq 2K_{1}\left( f,v,C\left( 
\mathbb{T}\right) \right) _{C\left( \mathbb{T}\right) }.  \label{DI3}
\end{equation}%
Firstly, let us prove the right hand side of (\ref{DI3}). For any $g\in
C^{1}\left( \mathbb{T}\right) $%
\begin{equation*}
\left\Vert f-\mathfrak{T}_{v}f\right\Vert _{C\left( \mathbb{T}\right)
}=\left\Vert f-g+g-\mathfrak{T}_{v}f\right\Vert _{C\left( \mathbb{T}\right) }
\end{equation*}%
\begin{equation*}
\leq \left\Vert f-g\right\Vert _{C\left( \mathbb{T}\right) }+\left\Vert g-%
\mathfrak{T}_{v}f\right\Vert _{C\left( \mathbb{T}\right) }
\end{equation*}%
\begin{equation*}
=\left\Vert f-g\right\Vert _{C\left( \mathbb{T}\right) }+\left\Vert g-%
\mathfrak{T}_{v}g+\mathfrak{T}_{v}g-\mathfrak{T}_{v}f\right\Vert _{C\left( 
\mathbb{T}\right) }
\end{equation*}%
\begin{equation*}
\leq \left\Vert f-g\right\Vert _{C\left( \mathbb{T}\right) }+\left\Vert g-%
\mathfrak{T}_{v}g\right\Vert _{C\left( \mathbb{T}\right) }+\left\Vert 
\mathfrak{T}_{v}g-\mathfrak{T}_{v}f\right\Vert _{C\left( \mathbb{T}\right) }
\end{equation*}%
\begin{equation*}
=\left\Vert f-g\right\Vert _{C\left( \mathbb{T}\right) }+\left\Vert g-%
\mathfrak{T}_{v}g\right\Vert _{C\left( \mathbb{T}\right) }+\left\Vert 
\mathfrak{T}_{v}\left( g-f\right) \right\Vert _{C\left( \mathbb{T}\right) }
\end{equation*}%
\begin{equation*}
\leq 2\left\Vert f-g\right\Vert _{C\left( \mathbb{T}\right) }+\frac{v}{2}%
\left\Vert g^{\prime }\right\Vert _{C\left( \mathbb{T}\right) }\leq
2K_{1}\left( f,v,C\left( \mathbb{T}\right) \right) _{C\left( \mathbb{T}%
\right) }.
\end{equation*}%
For the left hand side of inequality (\ref{DI3}) we need inequalities%
\begin{equation}
\left\Vert f-\mathfrak{T}_{v}^{2}f\right\Vert _{C\left( \mathbb{T}\right)
}\leq 2\left\Vert f-\mathfrak{T}_{v}f\right\Vert _{C\left( \mathbb{T}\right)
},  \label{uc}
\end{equation}%
\begin{equation}
v\left\Vert \left( \frac{d}{dx}\right) ^{2}\mathfrak{T}_{v}^{2}f\right\Vert
_{C\left( \mathbb{T}\right) }\leq 34\left\Vert f-\mathfrak{T}%
_{v}f\right\Vert _{C\left( \mathbb{T}\right) }.  \label{dort}
\end{equation}

We prove (\ref{uc}).%
\begin{equation*}
\left\Vert f-\mathfrak{T}_{v}^{2}f\right\Vert _{C\left( \mathbb{T}\right)
}=\left\Vert f-\mathfrak{T}_{v}f+\mathfrak{T}_{v}f-\mathfrak{T}%
_{v}^{2}f\right\Vert _{C\left( \mathbb{T}\right) }
\end{equation*}%
\begin{equation*}
\leq \left\Vert f-\mathfrak{T}_{v}f\right\Vert _{C\left( \mathbb{T}\right)
}+\left\Vert \mathfrak{T}_{v}f-\mathfrak{T}_{v}^{2}f\right\Vert _{C\left( 
\mathbb{T}\right) }\leq 2\left\Vert f-\mathfrak{T}_{v}f\right\Vert _{C\left( 
\mathbb{T}\right) }.
\end{equation*}%
Now we consider inequality (\ref{dort}). In (\ref{vor}) we replace $f$ by $%
\mathfrak{T}_{v}^{2}f$ and to obtain%
\begin{equation*}
\left\Vert \mathfrak{T}_{v}^{2}f\left( x\right) -\mathfrak{T}_{v}\mathfrak{T}%
_{v}^{2}f\left( x\right) +\frac{v}{2}\frac{d}{dx}\mathfrak{T}_{v}^{2}f\left(
x\right) \right\Vert _{C\left( \mathbb{T}\right) }\leq \frac{v^{2}}{6}%
\left\Vert \frac{d^{2}}{dx^{2}}\mathfrak{T}_{v}^{2}f\right\Vert _{C\left( 
\mathbb{T}\right) }.
\end{equation*}%
On the other hand, by (\ref{iki}),%
\begin{equation*}
\left\Vert \frac{d^{2}}{dx^{2}}\mathfrak{T}_{v}^{2}f\right\Vert _{C\left( 
\mathbb{T}\right) }\leq \frac{2}{v}\left\Vert \frac{d}{dx}\mathfrak{T}%
_{v}f\right\Vert _{C\left( \mathbb{T}\right) }\leq \frac{2}{v}\left\{
\left\Vert \frac{d}{dx}\mathfrak{T}_{v}^{2}f\right\Vert _{C\left( \mathbb{T}%
\right) }+\left\Vert \frac{d}{dx}\mathfrak{T}_{v}\left( \mathfrak{T}%
_{v}f-f\right) \right\Vert _{C\left( \mathbb{T}\right) }\right\}
\end{equation*}%
\begin{equation*}
\leq \frac{2}{v}\left\Vert \frac{d}{dx}\mathfrak{T}_{v}^{2}f\right\Vert
_{C\left( \mathbb{T}\right) }+\frac{4}{v^{2}}\left\Vert \mathfrak{T}%
_{v}f-f\right\Vert _{C\left( \mathbb{T}\right) }.
\end{equation*}%
Hence,%
\begin{equation*}
\frac{v}{2}\left\Vert \frac{d}{dx}\mathfrak{T}_{v}^{2}f\right\Vert _{C\left( 
\mathbb{T}\right) }\leq \left\Vert \mathfrak{T}_{v}^{2}f-\mathfrak{T}_{v}%
\mathfrak{T}_{v}^{2}f-\frac{v}{2}\frac{d}{dx}\mathfrak{T}_{v}^{2}f\right%
\Vert _{C\left( \mathbb{T}\right) }+\left\Vert \mathfrak{T}_{v}^{2}f-%
\mathfrak{T}_{v}\mathfrak{T}_{v}^{2}f\right\Vert _{C\left( \mathbb{T}\right)
}
\end{equation*}%
\begin{equation*}
\leq \frac{v^{2}}{6}\left\Vert \frac{d^{2}}{dx^{2}}\mathfrak{T}%
_{v}^{2}f\right\Vert _{C\left( \mathbb{T}\right) }+\left\Vert \mathfrak{T}%
_{v}^{2}f-\mathfrak{T}_{v}\mathfrak{T}_{v}^{2}f\right\Vert _{C\left( \mathbb{%
T}\right) }
\end{equation*}%
\begin{eqnarray*}
&\leq &\frac{v^{2}}{6}\frac{2}{v}\left\{ \left\Vert \frac{d}{dx}\mathfrak{T}%
_{v}^{2}f\right\Vert _{C\left( \mathbb{T}\right) }+\frac{2}{v}\left\Vert 
\mathfrak{T}_{v}f-f\right\Vert _{C\left( \mathbb{T}\right) }\right\}
+\left\Vert \mathfrak{T}_{v}^{2}f-f\right\Vert _{C\left( \mathbb{T}\right) }
\\
&&+\left\Vert \mathfrak{T}_{v}\left( \mathfrak{T}_{v}^{2}f-f\right)
\right\Vert _{C\left( \mathbb{T}\right) }+\left\Vert \mathfrak{T}%
_{v}f-f\right\Vert _{C\left( \mathbb{T}\right) }.
\end{eqnarray*}%
Then%
\begin{equation*}
\frac{v}{6}\left\Vert \frac{d}{dx}\mathfrak{T}_{v}^{2}f\right\Vert _{C\left( 
\mathbb{T}\right) }\leq \frac{17}{3}\left\Vert \mathfrak{T}%
_{v}f-f\right\Vert _{C\left( \mathbb{T}\right) },
\end{equation*}%
\begin{equation*}
v\left\Vert \frac{d}{dx}\mathfrak{T}_{v}^{2}f\right\Vert _{C\left( \mathbb{T}%
\right) }\leq 34\left\Vert \mathfrak{T}_{v}f-f\right\Vert _{C\left( \mathbb{T%
}\right) }.
\end{equation*}%
To finish proof of the left hand side of inequality (\ref{DI3}) we proceed as%
\begin{equation*}
K_{1}\left( f,v,C\left( \mathbb{T}\right) \right) _{C\left( \mathbb{T}%
\right) }\leq \left\Vert f-\mathfrak{T}_{v}^{2}f\right\Vert _{C\left( 
\mathbb{T}\right) }+v\left\Vert \frac{d}{dx}\mathfrak{T}_{v}^{2}f\right\Vert
_{C\left( \mathbb{T}\right) }
\end{equation*}%
\begin{equation*}
\leq 36\left\Vert T_{v}f-f\right\Vert _{C\left( \mathbb{T}\right) }.
\end{equation*}
\end{proof}

\begin{proof}[Proof of Lemma \protect\ref{bukun}]
Let $0<h\leq v<\infty $, $u\in \mathbb{T}$, $\gamma \in A_{p}$, $1\leq
p<\infty $ and $f\in L_{p,\gamma }$. By (\ref{eq1}) and Lemma \ref{Fu}%
\begin{equation*}
\max_{u\in \mathbb{T}}\left\vert F_{\left( \mathbb{I}-\mathfrak{T}%
_{h}\right) f,G}\left( u\right) \right\vert =\max_{u\in \mathbb{T}%
}\left\vert \int\nolimits_{\mathbb{T}}\left( \mathbb{I}-\mathfrak{T}%
_{h}\right) f\left( x+u\right) \left\vert G\left( x\right) \right\vert
dx\right\vert 
\end{equation*}%
\begin{equation*}
=\max_{u\in \mathbb{T}}\left\vert \left( \mathbb{I}-\mathfrak{T}_{h}\right)
\int\nolimits_{\mathbb{T}}f\left( x+u\right) \left\vert G\left( x\right)
\right\vert dx\right\vert 
\end{equation*}%
\begin{equation*}
=\max_{u\in \mathbb{T}}\left\vert \left( \mathbb{I}-\mathfrak{T}_{h}\right)
F_{f}\left( u\right) \right\vert \overset{\text{(\ref{eq1})}}{\leq }%
72\max_{u\in \mathbb{T}}\left\vert \left( \mathbb{I}-\mathfrak{T}_{v}\right)
F_{f}\left( u\right) \right\vert =
\end{equation*}%
\begin{equation*}
=72\max_{u\in \mathbb{T}}\left\vert \left( \mathbb{I}-\mathfrak{T}%
_{v}\right) \int\nolimits_{\mathbb{T}}f\left( x+u\right) \left\vert G\left(
x\right) \right\vert dx\right\vert \leq 
\end{equation*}%
\begin{equation*}
=72\max_{u\in \mathbb{T}}\left\vert \int\nolimits_{\mathbb{T}}\left( \mathbb{%
I}-\mathfrak{T}_{v}\right) f\left( x+u\right) \left\vert G\left( x\right)
\right\vert dx\right\vert 
\end{equation*}%
\begin{equation*}
=72\max_{u\in \mathbb{T}}\left\vert F_{\left( \mathbb{I}-\mathfrak{T}%
_{v}\right) f,G}\left( u\right) \right\vert .
\end{equation*}%
Now, Theorem \ref{tra} give (\ref{bukunn}).
\end{proof}

\begin{proof}[Proof of Theorem \protect\ref{kubu}]
(\ref{n}) is corollary of Theorem \ref{stekl op norm}. Since%
\begin{equation*}
\left( \mathbb{I}-\mathfrak{T}_{v}\right) ^{r+k}=\left( \mathbb{I}-\mathfrak{%
T}_{v}\right) ^{r}\left( \mathbb{I}-\mathfrak{T}_{v}\right) ^{k}
\end{equation*}%
(\ref{bk}) follows from 
\begin{equation*}
\Omega _{r+k}\left( f,v\right) _{p,\gamma }\leq \left( 1+24\mathbf{c}_{0}%
\mathbf{c}_{0}^{\prime \prime }\mathbf{c}_{0}^{\prime \prime \prime }\right)
^{r}\Omega _{k}\left( f,v\right) _{p,\gamma }.
\end{equation*}
\end{proof}

\subsection{Equivalence with Peetre's \textit{K}-functional}

One of the main results of this paper is the following theorem, which
contains an equivalence of $\Omega _{r}\left( f,t\right) _{p,\gamma }$ and $%
K_{r}\left( f,t,p,\gamma \right) $ as a function on $t$.

\begin{theorem}
\label{reaa} If $r\in \mathbb{N}$, $1\leq p<\infty $, $\gamma \in A_{p}$, $%
f\in L_{p,\gamma }$, then the equivalence%
\begin{equation}
C_{2}\Omega _{r}\left( f,t\right) _{p,\gamma }\leq K_{r}\left( f,t,p,\gamma
\right) \leq C_{3}\Omega _{r}\left( f,t\right) _{p,\gamma }\text{,\qquad }t>0
\label{real}
\end{equation}%
holds, where%
\begin{equation*}
C_{2}=2^{-1}\left( 1+24\mathbf{c}_{0}\mathbf{c}_{0}^{\prime \prime }\mathbf{c%
}_{0}^{\prime \prime \prime }\right) ^{-r},\quad C_{3}=2\max \left\{
C_{4}^{r},C_{5}^{r}\right\} ,
\end{equation*}%
\begin{eqnarray*}
C_{4} &=&1728\mathbf{c}_{0}\mathbf{c}_{0}^{\prime \prime }\mathbf{c}%
_{0}^{\prime \prime \prime }\left( \sum\limits_{j=0}^{r-1}\left( 24\mathbf{c}%
_{0}\mathbf{c}_{0}^{\prime \prime }\mathbf{c}_{0}^{\prime \prime \prime
}\right) ^{j}\right) , \\
C_{5} &=&48\mathbf{c}_{0}\mathbf{c}_{0}^{\prime \prime }\mathbf{c}%
_{0}^{\prime \prime \prime }\left( 37+146\ln 2^{36}\right) .
\end{eqnarray*}
\end{theorem}

\textbf{Note }When $\gamma \equiv 1$ and $p\in \left[ 1,\infty \right] $, $%
\Omega _{r}\left( f,t\right) _{p,1}$ in $L_{p}$ was considered in \cite{Dit}
and there it was proved that $\Omega _{r}\left( f,t\right) _{p,1}$ is
equivalent to $K_{r}\left( f,t,p,1\right) $ for $r\in \mathbb{N}$, and $t>0$%
. See also \cite{ai11}, \cite{bgs}, \cite{isrgu06}, \cite{j10,j07,j15}, \cite%
{kc16}.

Note that (\ref{real}) implies the following properties of $\Omega
_{r}\left( f,\cdot \right) _{p,\gamma }$.

\begin{corollary}
\label{cov}If $k\in \mathbb{N}$, $1\leq p<\infty $, $\gamma \in A_{p}$, $%
f\in L_{p,\gamma }$, then%
\begin{equation*}
\Omega _{k}\left( f,\lambda v\right) _{p,\gamma }\leq C_{6}\left( 1+\lfloor
\lambda \rfloor \right) ^{k}\Omega _{k}\left( f,v\right) _{p,\gamma }\text{%
,\quad }v,\lambda >0,
\end{equation*}%
and%
\begin{equation*}
\Omega _{k}\left( f,v\right) _{p,\gamma }v^{-k}\leq C_{7}\Omega _{k}\left(
f,\delta \right) _{p,\gamma }\delta ^{-k}\text{,\quad }0<\delta \leq v,
\end{equation*}%
where $\lfloor z\rfloor :=\max \left\{ y\in \mathbb{Z}:y\leq z\right\} $,
with 
\begin{equation*}
C_{6}=2\left( 1+24\mathbf{c}_{0}\mathbf{c}_{0}^{\prime \prime }\mathbf{c}%
_{0}^{\prime \prime \prime }\right) ^{r}C_{3}\text{,\quad }C_{7}=\frac{%
C_{3}C_{6}\left( 1+\lfloor \frac{v}{\delta }\rfloor \right) ^{k}}{C_{2}}.
\end{equation*}
\end{corollary}

We will need following lemmas.

\begin{lemma}
\label{bukunA} Let $0<v<\infty $, $1\leq p<\infty $, $\gamma \in A_{p}$ and $%
f\in L_{p,\gamma }$. Then%
\begin{equation*}
\left\Vert f-R_{v}f\right\Vert _{p,\gamma }\leq 24\mathbf{c}_{0}\mathbf{c}%
_{0}^{\prime \prime }\mathbf{c}_{0}^{\prime \prime \prime }\left\Vert \left(
I-\mathfrak{T}_{v}\right) f\right\Vert _{p,\gamma }.
\end{equation*}
\end{lemma}

\begin{remark}
Note that, the function $R_{v}f$ is absolutely continuous (\cite{sh13i}) and
differentiable a.e. on $\mathbb{T}$.
\end{remark}

\begin{lemma}
\label{deg} Let $0<v<\infty $, $\gamma \in A_{p}$, $1\leq p<\infty $ and $%
f\in W_{p,\gamma }^{1}$. Then%
\begin{equation*}
\frac{d}{dx}R_{v}f\left( x\right) =R_{v}\frac{d}{dx}f\left( x\right) \text{%
\quad and\quad }\frac{d}{dx}\mathfrak{T}_{v}f\left( x\right) =\mathfrak{T}%
_{v}\frac{d}{dx}f\left( x\right) \text{,}\quad \text{a.e. }x\in \mathbb{T}.
\end{equation*}
\end{lemma}

\begin{lemma}
\label{lm05} Let $0<v<\infty $, $\gamma \in A_{p}$, $1\leq p<\infty $ and $%
f\in L_{p,\gamma }$ be given. Then%
\begin{equation}
v\left\Vert \frac{d}{dx}R_{v}f(x)\right\Vert _{p,\gamma }\leq
C_{5}\left\Vert \left( \mathbb{I}-\mathfrak{T}_{v}\right) f\right\Vert
_{p,\gamma }.  \label{mera}
\end{equation}
\end{lemma}

We set $R_{v}^{r}f:=\left( R_{v}f\right) ^{r}.$

\begin{lemma}
\label{da} Let $0<v<\infty $, $r-1\in \mathbb{N}$, $1\leq p<\infty $, $%
\gamma \in A_{p}$ and $f\in L_{p,\gamma }$ be given. Then%
\begin{equation}
\frac{d^{r}}{dx^{r}}R_{v}^{r}f\left( x\right) =\frac{d}{dx}R_{v}\frac{d^{r-1}%
}{dx^{r-1}}R_{v}^{r-1}f\left( x\right) ,\quad x\in \mathbb{T}.  \label{res}
\end{equation}
\end{lemma}

We give the proof of required lemmas and Theorem \ref{reaa}.

\begin{proof}[Proof of Lemma \protect\ref{bukunA}]
If $f\in L_{p,\gamma }$, using generalized Minkowski's integral inequality
and Lemma \ref{bukun} we obtain%
\begin{equation*}
\left\Vert f-R_{v}f\right\Vert _{p,\gamma }=\left\Vert \frac{2}{v}%
\int\nolimits_{v/2}^{v}\left( \frac{1}{h}\int\nolimits_{0}^{h}\left( f\left(
x+t\right) -f\left( x\right) \right) dt\right) dh\right\Vert _{p,\gamma }
\end{equation*}%
\begin{equation*}
=\left\Vert \frac{2}{v}\int\nolimits_{v/2}^{v}\left( \mathfrak{T}_{h}f\left(
x\right) -f\left( x\right) \right) dh\right\Vert _{p,\gamma }\leq \frac{2}{v}%
\int\nolimits_{v/2}^{v}\left\Vert \mathfrak{T}_{v}f-f\right\Vert _{p,\gamma
}dh
\end{equation*}%
\begin{eqnarray}
&\leq &1728\mathbf{c}_{0}\mathbf{c}_{0}^{\prime \prime }\mathbf{c}%
_{0}^{\prime \prime \prime }\left\Vert \mathfrak{T}_{v}f-f\right\Vert
_{p,\gamma }\frac{2}{v}\int\nolimits_{v/2}^{v}dh  \label{b1} \\
&=&1728\mathbf{c}_{0}\mathbf{c}_{0}^{\prime \prime }\mathbf{c}_{0}^{\prime
\prime \prime }\left\Vert \left( I-\mathfrak{T}_{v}\right) f\right\Vert
_{p,\gamma }.  \notag
\end{eqnarray}
\end{proof}

\begin{proof}[Proof of Lemma \protect\ref{deg}]
The first result follows from%
\begin{equation*}
\frac{d}{dx}R_{v}f(x)=\frac{d}{dx}\left( \frac{2}{v}\int\nolimits_{v/2}^{v}%
\left( \frac{1}{h}\int\nolimits_{0}^{h}f\left( x+t\right) dt\right) dh\right)
\end{equation*}%
\begin{equation*}
=\frac{d}{dx}\left( \frac{2}{v}\int\nolimits_{v/2}^{v}\left( \frac{1}{h}%
\int\nolimits_{x}^{x+h}f\left( \tau \right) d\tau \right) dh\right)
\end{equation*}%
\begin{equation*}
=\left( \frac{2}{v}\int\nolimits_{v/2}^{v}\left( \frac{1}{h}%
\int\nolimits_{0}^{h}\frac{d}{dx}f\left( x+t\right) dt\right) dh\right)
=R_{v}\frac{d}{dx}f\left( x\right) .
\end{equation*}%
For the second one we find%
\begin{eqnarray*}
\frac{d}{dx}\mathfrak{T}_{v}f(x) &=&\frac{d}{dx}\left( \frac{1}{h}%
\int\nolimits_{0}^{h}f\left( x+t\right) dt\right) =\frac{d}{dx}\left( \frac{1%
}{h}\int\nolimits_{x}^{x+h}f\left( \tau \right) d\tau \right) \\
&=&\frac{1}{h}\int\nolimits_{x}^{x+h}\frac{d}{dx}f\left( \tau \right) d\tau =%
\mathfrak{T}_{v}\frac{d}{dx}f(x).
\end{eqnarray*}
\end{proof}

\begin{proof}[Proof of Lemma \protect\ref{lm05}]
Using%
\begin{equation*}
\left\Vert F_{\delta \left( \mathfrak{R}_{\delta }f\right) ^{\prime
},G}\right\Vert _{C\left( T \right) }=\left\Vert \delta
\left( F_{\left( \mathfrak{R}_{\delta }f\right) ,G}\right) ^{\prime
}\right\Vert _{C\left( T \right) }=\delta \left\Vert \left( 
\mathfrak{R}_{\delta }(F_{f,G})\right) ^{\prime }\right\Vert _{C\left( T\right) }
\end{equation*}%
\begin{equation*}
\leq \cdots \leq 2\left( 37+146\ln 2^{36}\right) \left\Vert \left(
I-T_{\delta }\right) (F_{f,G})\right\Vert _{C\left( T \right) }
\end{equation*}%
\begin{equation*}
=2\left( 37+146\ln 2^{36}\right) \left\Vert (F_{\left( I-T_{\delta }\right)
f,G})\right\Vert _{C\left( T \right) }
\end{equation*}%
we conclude from TR that%
\begin{equation*}
\delta \left\Vert (\mathfrak{R}_{\delta }f)^{\prime }\right\Vert _{p,\gamma
}\leq 48\mathbf{c}_{0}\mathbf{c}_{0}^{\prime \prime }\mathbf{c}_{0}^{\prime
\prime \prime }\left( 37+146\ln 2^{36}\right) \left\Vert \left( I-T_{\delta
}\right) f\right\Vert _{p,\gamma }.
\end{equation*}
\end{proof}

\begin{proof}[Proof of Lemma \protect\ref{da}]
For $r=2$, by Lemma \ref{deg},%
\begin{eqnarray*}
\frac{d^{2}}{dx^{2}}R_{v}^{2}f &=&\frac{d}{dx}\frac{d}{dx}R_{v}R_{v}f=\frac{d%
}{dx}\frac{d}{dx}R_{v}\Psi \qquad \left[ \Psi :=R_{v}f\right] \\
&=&\frac{d}{dx}R_{v}\frac{d}{dx}\Psi =\frac{d}{dx}R_{v}\frac{d}{dx}R_{v}f
\end{eqnarray*}%
and the result (\ref{res}) follows. For $r=3$, by Lemma \ref{deg},%
\begin{equation*}
\frac{d^{3}}{dx^{3}}R_{v}^{3}f=\frac{d}{dx}\frac{d^{2}}{dx^{2}}%
R_{v}^{2}R_{v}f=\frac{d}{dx}\frac{d^{2}}{dx^{2}}R_{v}^{2}\Psi =\frac{d}{dx}%
\frac{d}{dx}R_{v}\frac{d}{dx}R_{v}\Psi
\end{equation*}%
\begin{equation*}
=\frac{d}{dx}\frac{d}{dx}R_{v}\frac{d}{dx}R_{v}^{2}f=\frac{d}{dx}R_{v}\frac{d%
}{dx}\frac{d}{dx}R_{v}^{2}f=\frac{d}{dx}R_{v}\frac{d^{2}}{dx^{2}}R_{v}^{2}f
\end{equation*}%
and (\ref{res}) holds. Let (\ref{res}) holds for $k\in \mathbb{N}$:%
\begin{equation}
\frac{d^{k}}{dx^{k}}R_{v}^{k}f=\frac{d}{dx}R_{v}\frac{d^{k-1}}{dx^{k-1}}%
R_{v}^{k-1}f.  \label{bula}
\end{equation}%
Then, for $k+1$, (\ref{bula}) and Lemma \ref{deg} implies that%
\begin{equation*}
\frac{d^{k+1}}{dx^{k+1}}R_{v}^{k+1}f=\frac{d}{dx}\frac{d^{k}}{dx^{k}}%
R_{v}^{k}R_{v}f=\frac{d}{dx}\frac{d^{k}}{dx^{k}}R_{v}^{k}\Psi =\frac{d}{dx}%
\frac{d}{dx}R_{v}\frac{d^{k-1}}{dx^{k-1}}R_{v}^{k-1}\Psi
\end{equation*}%
\begin{equation*}
=\frac{d}{dx}\frac{d}{dx}R_{v}\frac{d^{k-1}}{dx^{k-1}}R_{v}^{k}f=\frac{d}{dx}%
R_{v}\frac{d}{dx}\frac{d^{k-1}}{dx^{k-1}}R_{v}^{k}f=\frac{d}{dx}R_{v}\frac{%
d^{k}}{dx^{k}}R_{v}^{k}f.
\end{equation*}
\end{proof}

\begin{proof}[Proof of Theorem \protect\ref{reaa}]
For $r=1,2,3,\ldots $ we consider the operator (\cite{ra11u})%
\begin{equation*}
A_{\delta }^{r}:=\mathbb{I}-\left( \mathbb{I}-R_{v}^{r}\right) ^{r}.
\end{equation*}%
From the identity $\mathbb{I}-R_{v}^{r}=\left( \mathbb{I}-R_{v}\right)
\sum\nolimits_{j=0}^{r-1}R_{v}^{j}$ we find%
\begin{equation*}
\left\Vert \left( \mathbb{I}-R_{v}^{r}\right) g\right\Vert _{p,\gamma }\leq
\left( \sum\limits_{j=0}^{r-1}\left( 1728\mathbf{c}_{0}\mathbf{c}%
_{0}^{\prime \prime }\mathbf{c}_{0}^{\prime \prime \prime }\right)
^{j}\right) \left\Vert \left( \mathbb{I}-R_{v}\right) g\right\Vert
_{p,\gamma }
\end{equation*}%
\begin{eqnarray}
&\leq &\left( 1728\mathbf{c}_{0}\mathbf{c}_{0}^{\prime \prime }\mathbf{c}%
_{0}^{\prime \prime \prime }\left( \sum\limits_{j=0}^{r-1}\left( 1728\mathbf{%
c}_{0}\mathbf{c}_{0}^{\prime \prime }\mathbf{c}_{0}^{\prime \prime \prime
}\right) ^{j}\right) \right) \left\Vert \left( \mathbb{I}-\mathfrak{T}%
_{v}\right) g\right\Vert _{p,\gamma }  \label{fdd} \\
&=&C_{4}\left\Vert \left( \mathbb{I}-\mathfrak{T}_{v}\right) g\right\Vert
_{p,\gamma }  \notag
\end{eqnarray}%
when $0<v<\infty $, $\gamma \in A_{p},1\leq p<\infty $ and $g\in L_{p,\gamma
}$. Since $\left\Vert f-A_{v}^{r}f\right\Vert _{p,\gamma }=\left\Vert \left( 
\mathbb{I}-R_{v}^{r}\right) ^{r}f\right\Vert _{p,\gamma }$, recursive
procedure gives%
\begin{equation*}
\left\Vert f-A_{v}^{r}f\right\Vert _{p,\gamma }=\left\Vert \left( \mathbb{I}%
-R_{v}^{r}\right) ^{r}f\right\Vert _{p,\gamma }=\left\Vert \left( \mathbb{I}%
-R_{v}^{r}\right) \left( \mathbb{I}-R_{v}^{r}\right) ^{r-1}f\right\Vert
_{p,\gamma }
\end{equation*}%
\begin{equation*}
\overset{\left( \text{\ref{fdd}}\right) }{\leq }C_{4}\left\Vert \left( 
\mathbb{I}-\mathfrak{T}_{v}\right) \left( \mathbb{I}-R_{v}^{r}\right)
^{r-1}f\right\Vert _{p,\gamma }\overset{\left( \text{\ref{fdd}}\right) }{%
\leq }\cdots 
\end{equation*}%
\begin{equation*}
\overset{\left( \text{\ref{fdd}}\right) }{\leq }C_{4}^{2}\left\Vert \left( 
\mathbb{I}-\mathfrak{T}_{v}\right) ^{2}\left( \mathbb{I}-R_{v}^{r}\right)
^{r-2}f\right\Vert _{p,\gamma }\overset{\left( \text{\ref{fdd}}\right) }{%
\leq }\cdots 
\end{equation*}%
\begin{equation*}
\overset{\left( \text{\ref{fdd}}\right) }{\leq }C_{4}^{3}\left\Vert \left( 
\mathbb{I}-\mathfrak{T}_{v}\right) ^{3}\left( \mathbb{I}-R_{v}^{r}\right)
^{r-3}f\right\Vert _{p,\gamma }\overset{\left( \text{\ref{fdd}}\right) }{%
\leq }\cdots \overset{\left( \text{\ref{fdd}}\right) }{\leq }%
C_{4}^{r}\left\Vert \left( \mathbb{I}-\mathfrak{T}_{v}\right)
^{r}f\right\Vert _{p,\gamma }.
\end{equation*}%
On the other hand, using (\ref{mera}), Lemmas \ref{da} and \ref{deg},
recursively,%
\begin{equation*}
v^{r}\left\Vert \frac{d^{r}}{dx^{r}}R_{v}^{r}f\right\Vert _{p,\gamma
}=v^{r-1}v\left\Vert \frac{d}{dx}R_{v}\frac{d^{r-1}}{dx^{r-1}}%
R_{v}^{r-1}f\right\Vert _{p,\gamma }
\end{equation*}%
\begin{equation*}
\leq C_{5}v^{r-1}\left\Vert \left( \mathbb{I}-\mathfrak{T}_{v}\right) \frac{%
d^{r-1}}{dx^{r-1}}R_{v}^{r-1}f\right\Vert _{p,\gamma }
\end{equation*}%
\begin{equation*}
=C_{5}v^{r-2}v\left\Vert \frac{d^{r-1}}{dx^{r-1}}R_{v}^{r-1}\left( \mathbb{I}%
-\mathfrak{T}_{v}\right) f\right\Vert _{p,\gamma }
\end{equation*}%
\begin{equation*}
=C_{5}v^{r-2}v\left\Vert \frac{d}{dx}R_{v}\frac{d^{r-2}}{dx^{r-2}}%
R_{v}^{r-2}\left( \mathbb{I}-\mathfrak{T}_{v}\right) f\right\Vert _{p,\gamma
}
\end{equation*}%
\begin{equation*}
\leq C_{5}^{2}v^{r-2}\left\Vert \left( \mathbb{I}-\mathfrak{T}_{v}\right) 
\frac{d^{r-2}}{dx^{r-2}}R_{v}^{r-2}\left( \mathbb{I}-\mathfrak{T}_{v}\right)
f\right\Vert _{p,\gamma }
\end{equation*}%
\begin{equation*}
=C_{5}^{2}v^{r-2}\left\Vert \frac{d^{r-2}}{dx^{r-2}}R_{v}^{r-2}\left( 
\mathbb{I}-\mathfrak{T}_{v}\right) ^{2}f\right\Vert _{p,\gamma }
\end{equation*}%
\begin{equation*}
\leq \cdots \leq C_{5}^{r-1}v\left\Vert \frac{d}{dx}R_{v}\left( \mathbb{I}-%
\mathfrak{T}_{v}\right) ^{r-1}f\right\Vert _{p,\gamma }\leq
C_{5}^{r}\left\Vert \left( \mathbb{I}-\mathfrak{T}_{v}\right)
^{r}f\right\Vert _{p,\gamma }.
\end{equation*}%
Thus%
\begin{equation*}
K_{r}\left( f,v,p,\gamma \right) _{p,\gamma }\leq \left\Vert
f-A_{v}^{r}f\right\Vert _{p,\gamma }+v^{r}\left\Vert \frac{d^{r}}{dx^{r}}%
A_{v}^{r}f\left( x\right) \right\Vert _{p,\gamma }
\end{equation*}%
\begin{equation*}
\leq 2\max \left\{ C_{4}^{r},C_{5}^{r}\right\} \left\Vert \left( \mathbb{I}-%
\mathfrak{T}_{v}\right) ^{r}f\right\Vert _{p,\gamma }.
\end{equation*}%
For the reverse of the last inequality, when $g\in W_{p,\gamma }^{r}$, (from
Lemma \ref{lem1})%
\begin{equation*}
\Omega _{r}\left( f,v\right) _{p,\gamma }\leq \left( 1+24\mathbf{c}_{0}%
\mathbf{c}_{0}^{\prime \prime }\mathbf{c}_{0}^{\prime \prime \prime }\right)
^{r}\left\Vert f-g\right\Vert _{p,\gamma }+\Omega _{r}\left( g,v\right)
_{p,\gamma }
\end{equation*}%
\begin{equation}
\leq \left( 1+24\mathbf{c}_{0}\mathbf{c}_{0}^{\prime \prime }\mathbf{c}%
_{0}^{\prime \prime \prime }\right) ^{r}\left\Vert f-g\right\Vert _{p,\gamma
}+(12\mathbf{c}_{0}\mathbf{c}_{0}^{\prime \prime }\mathbf{c}_{0}^{\prime
\prime \prime })^{r}v^{r}\left\Vert g^{\left( r\right) }\right\Vert
_{p,\gamma },  \label{mn}
\end{equation}%
and taking infimum on $g\in W_{p,\gamma }^{r}$ in (\ref{mn}) we get%
\begin{equation*}
\Omega _{r}\left( f,v\right) _{p,\gamma }\leq C_{2}^{-1}K_{r}\left(
f,v,p,\gamma \right) _{p,\gamma }
\end{equation*}%
and hence (\ref{real}) holds.
\end{proof}

\begin{proof}[Proof of Theorem \protect\ref{cov}]
If $r\in \mathbb{N}$, $1\leq p<\infty $, $\gamma \in A_{p}$, $f\in
L_{p,\gamma }$, then%
\begin{equation*}
\Omega _{r}\left( f,\lambda v\right) _{p,\gamma }\leq 2\left( 1+24\mathbf{c}%
_{0}\mathbf{c}_{0}^{\prime \prime }\mathbf{c}_{0}^{\prime \prime \prime
}\right) ^{r}C_{3}\left( 1+\lfloor \lambda \rfloor \right) ^{r}\Omega
_{r}\left( f,v\right) _{p,\gamma },\quad v,\lambda >0.
\end{equation*}%
and hence%
\begin{equation*}
\Omega _{k}\left( f,v\right) _{p,\gamma }v^{-k}\leq C_{7}\Omega _{k}\left(
f,\delta \right) _{p,\gamma }\delta ^{-k},\quad 0<\delta \leq v.
\end{equation*}
\end{proof}

ADDRESS: BALIKES\.{I}R UNIVERSITY, FACULTY OF ARTS AND SCIENCES, DEPARTMENT
OF MATHEMATICS, \c{C}A\u{G}I\c{S} YERLE\c{S}KES\.{I}, 10145, BALIKES\.{I}R,
TURKEY.

rakgun@balikesir.edu.tr

\end{document}